\documentclass[a4paper,oneside,11pt]{amsart}

\reversemarginpar

\usepackage[colorlinks,hyperindex,linkcolor=blue,urlcolor=black,pdftitle={Invertibility of functions of operators and existence of hyperinvariant subspaces},pdfauthor={Maria F. Gamal'}]
{hyperref}

\usepackage{amsmath}
\usepackage{amsfonts}
\usepackage{amssymb}
\usepackage{amsthm}

\usepackage[english]{babel}
\usepackage{enumerate}

\usepackage{graphicx}
\usepackage{color}

\usepackage[abbrev]{amsrefs}

\numberwithin{equation}{section}

\theoremstyle{plain}
\newtheorem{theorem}{Theorem}[section]
\newtheorem{lemma}[theorem]{Lemma}
\newtheorem{corollary}[theorem]{Corollary}
\newtheorem{proposition}[theorem]{Proposition}

\theoremstyle{definition}
\newtheorem{remark}[theorem]{Remark}
\newtheorem{example}[theorem]{Example}

\begin{document}

\title[Example of quasianalytic contraction]{Example of quasianalytic contraction whose spectrum is a proper subarc of the unit circle}

\author{Maria F. Gamal'}
\address{
 St. Petersburg Branch\\ V. A. Steklov Institute 
of Mathematics\\
 Russian Academy of Sciences\\ Fontanka 27, St. Petersburg\\ 
191023, Russia  
}
\email{gamal@pdmi.ras.ru}




\begin{abstract}A partial answer on {\cite[Question 2]{ks15}} is given. Namely, an operator $R$ similar to a quasianalytic contraction 
whose quasianalytic spectral set is equal to its spectrum and  is a proper subarc of the unit circle is constructed, 
but no estimates of $\|R^{-1}\|$ is given.
\end{abstract}

\maketitle

\section{Introduction}

Let $\mathcal H$ be a (complex, separable) Hilbert space, and let 
 $T$ be a   (linear, bounded) operator on $\mathcal H$.
The lattice of all (closed) subspaces $\mathcal E$ of $\mathcal H$ such that $T\mathcal E\subset\mathcal E$ is called 
the \emph{invariant subspace lattice} of $T$ and is denoted by $\operatorname{Lat}T$. 
  The \emph{commutant} $\{T\}'$ is the set of all operators $A$ on  $\mathcal H$ such that $AT=TA$. Recall 
that $\{T\}'$ is an algebra closed in the weak operator topology. 
The lattice of all  subspaces $\mathcal E$ of $\mathcal H$ such that $A\mathcal E\subset\mathcal E$ 
for every $A\in\{T\}'$ is called 
the \emph{hyperinvariant subspace lattice} of $T$ and is denoted by $\operatorname{Hlat}T$; the subspaces $\mathcal E$ are 
called \emph{hyperinvariant subspaces} of $T$. 

Let $\mathcal H$  and $\mathcal K$ be two Hilbert space. Denote by $\mathcal L(\mathcal H, \mathcal K)$ 
the space of all  (linear, bounded) transformations acting from $\mathcal H$ to  $\mathcal K$. 
Set $\mathcal L(\mathcal H)=\mathcal L(\mathcal H, \mathcal H)$, then $\mathcal L(\mathcal H)$ is 
 the algebra of all (linear, bounded) operators acting on $\mathcal H$. Let $T\in\mathcal L(\mathcal H)$,
$R\in\mathcal L(\mathcal K)$, $X\in\mathcal L(\mathcal H, \mathcal K)$ be such that $XT=RX$. If $X$ is \emph{invertible}, that is, 
$X^{-1}\in\mathcal L(\mathcal K, \mathcal H)$, then $T$ and $R$ are called \emph{similar}. If, in addition, 
$X$ is a unitary transformation, then $T$ and $R$ are called \emph{unitarily equivalent}.

An operator  $T\in\mathcal L(\mathcal H)$ is called \emph{power bounded}, if $\sup_{n\in\mathbb N}\|T^n\|<\infty$.
 A power bounded operator $T$ is 
\emph{of class} $C_{1\cdot}$, if $\inf_{n\in\mathbb N}\|T^n x\|>0$ for every $0\neq x\in\mathcal H$, 
and is \emph{of class} $C_{0\cdot}$, if $\lim_n\|T^n x\|=0$ for every $ x\in\mathcal H$. 
$T$ is \emph{of class} $C_{\cdot a}$, if $T^*$ is of class $C_{a\cdot}$, and $T$ is \emph{of class} $C_{ab}$, 
if $T$ is of classes $C_{a\cdot}$ and $C_{\cdot b}$, $a,b=0,1$.

An operator   $T\in\mathcal L(\mathcal H)$ is called \emph{polynomially bounded}, if there exists a constant $M$
such that 
\begin{align*} \|p(T)\|\leq M\max\{|p(z)|\ :\ z\in\operatorname{clos}\mathbb D\}\\
\text{ for every (analytic) polynomial } p,\end{align*} where $\mathbb D$ is the open unit disc.
For a natural number $n$ a $n\times n$ matrix can be regarded as an operator on 
$\ell_n^2$, its norm is denoted by the symbol $\|\cdot\|_{\mathcal L(\ell_n^2)}$. 
For a family of polynomials $[p_{ij}]_{i,j=1}^n$ put
$$\| [p_{ij}]_{i,j=1}^n\|_{H^\infty(\ell_n^2)}= 
\sup\{\| [p_{ij}(z)]_{i,j=1}^n\|_{\mathcal L(\ell_n^2)}, \ z\in\operatorname{clos}\mathbb D\}.$$
For  $T\in\mathcal L(\mathcal H)$  and   a  family   of  polynomials   
$[p_{ij}]_{i,j=1}^n$ \
the operator  $$[p_{ij}(T)]_{i,j=1}^n\in\mathcal L(\oplus_{j=1}^n \mathcal H)$$ is defined.
$T$ is called \emph{completely polynomially bounded}, 
if there exists a constant $M$ such that
\begin{equation} \label{1.2}
\begin{aligned} \| [p_{ij}(T)]_{i,j=1}^n\|\leq
M\| [p_{ij}]_{i,j=1}^n\|_{H^\infty(\ell_n^2)} \\ 
\text{ for every family of polynomials } [p_{ij}]_{i,j=1}^n. \end{aligned}
\end{equation}

An operator $T$ is called a {\it contraction} if $\|T\|\leq 1$.
The following criterion for an operator to be similar to a contraction is proved 
in \cite{pau}:

\emph{An operator $T$ is similar to a contraction if and only if 
$T$ is completely polynomially bounded. }

We recall some definitions and  results on  unitary asymptotes and quasianalytic operators. 
For references on unitary asymptotes  see {\cite[Ch. IX.1]{nfbk}}, \cite{ker1}, \cite{ker4}, \cite{kerntuple}, 
for quasianalytic operators see also \cite{est},  \cite{ks14}, \cite{ks15}, \cite{gamstudia}, \cite{gam19}; 
see also references therein.
 
A pair $(X,U)$, where $U$ is a unitary operator, and $X$ is a 
 transformation such that $XT=UX$, 
 is called a \emph{unitary asymptote} of an operator $T$, if for any other pair $(Y,V)$, 
where $V$ is a unitary operator, and $Y$ is a  transformation such that $YT=VY$, there exists 
a unique transformation $Z$ such that $ZU=VZ$ and $Y=ZX$. Two pairs $(X,U)$ and $(X_1,U_1)$,
 where $U$ and $U_1$ are unitary operators, and $X$ and $X_1$ are  transformations such that $XT=UX$ and $X_1T=U_1X_1$, 
are  \emph{equivalent}, if there exists an \emph{invertible} transformation $Z$
such that $ZU=U_1Z$ and $X_1=ZX$. It follows from the definition  that 
 a unitary asymptote of $T$ is defined up to equivalence. 
If two operators are similar and one of them has a unitary asymptote, then other has  a unitary asymptote, too, 
and their unitary asymptotes are equivalent. 
If $ZU=U_1Z$ for an invertible transformation $Z$ 
and unitary operators $U$ and $U_1$, then $U$ and $U_1$ are unitarily equivalent 
({\cite[Proposition II.3.4]{nfbk}},  {\cite[Proposition II.10.6]{conw2}},  {\cite[Proposition 1.5]{rara}}). 

Let $T\in\mathcal L(\mathcal H)$ have a unitary asymptote $(X,U)$, where $U\in\mathcal L(\mathcal K)$. 
Then  there exists the mapping
$$\gamma_T\colon \{T\}'\to\{U\}', \ \ \gamma_T(A)=D, $$
 where $ D\in\{U\}'$ is  a unique operator  such that $XA=DX$, and $\gamma_T$ is a unital algebra-homomorphism.
Furthermore, 
\begin{equation}\label{sigmagamma}
\sigma(\gamma_T(A))\subset\sigma(A) \ \ \text{ for every }A\in\{T\}'. \end{equation} 
For $ \mathcal E\subset \mathcal K$ set 
$$X^{-1}\mathcal E =\{x\in\mathcal H\ :\ Xx\in\mathcal E \}.$$
Then $X^{-1}\mathcal E\in\operatorname{Hlat}T$ for every $\mathcal E\in\operatorname{Hlat}U$. 
\emph{We will assume that $\mathcal K\neq\{0\}$.} 
It is well khown that  $\operatorname{Hlat}U\neq\{\{0\},\mathcal K\}$, if $U$ is not the multiplication by a unimodular constant on 
 $\mathcal K$. But it is possible that 
\begin{equation}\label{defquasi} X^{-1}\mathcal E =\{0\} 
\ \text{ for every } \mathcal E\in\operatorname{Hlat}U \text{ such that }  \mathcal E\neq\mathcal K.
\end{equation}
Such an operator $T$ is called \emph{quasianalytic}. 

Denote by $\mathbf{m}$ the normalized linear measure on the unit circle $\mathbb T$. 
 For a measurable (with respect to $\mathbf{m}$) set $\sigma\subset\mathbb T$ denote by $U_\sigma$ the operator of multiplication by 
the independent variable on $L^2(\sigma):=L^2(\sigma, \mathbf{m})$. It is well known that $U_\sigma$ is cyclic, 
$$ \{U_\sigma\}'=\{\eta(U_\sigma)\ : \ \eta\in L^\infty(\sigma, \mathbf{m})\},$$
where $\eta(U_\sigma)$ is the operator of multiplication by $\eta$, and 
$$\operatorname{Hlat}U_\sigma = \{L^2(\tau)\ : \ \tau\subset\sigma\},$$
(where $\tau$ are  measurable with respect to $\mathbf{m}$). 

Let  $\sigma\subset\mathbb T$  be a measurable (with respect to $\mathbf{m}$) set, 
and let an operator $T$  have a unitary asymptote $(X,U_\sigma)$.  
Then for every $A\in\{T\}'$ there exists a function $\eta\in L^\infty(\sigma, \mathbf{m})=:L^\infty(\sigma)$ such that $\gamma_T(A)=\eta(U_\sigma)$. 
The mapping
\begin{equation}\label{gamma}
\widehat\gamma_T\colon \{T\}'\to L^\infty(\sigma), \ \ \widehat\gamma_T(A)=\eta, \end{equation}
is a unital algebra-homomorphism, and $\widehat\gamma_T$ does not depend of the choice of $X$. 
Furthermore, $\widehat\gamma_T(T)=\chi$, where $\chi(z)=z$, $z\in\mathbb T$. 
The range $\widehat\gamma_T(\{T\}')$ is called the \emph{functional commutant} of $T$, see \cite{ks14} and references therein. 

Every power bounded operator $T$ has a unitary asymptote, and if $T$ of class $C_{1\cdot}$, then $\gamma_T$ is injective. 
If, in addition,  a unitary operator from the unitary asymptote (\emph{which also will be called the unitary asymptote}) of $T$ is $U_\sigma$ for some $\sigma\subset\mathbb T$, 
then $\widehat\gamma_T$ is injective, too. Therefore, $\{T\}'$ is an abelian algebra, 
because $L^\infty(\sigma)$ is an abelian algebra. Moreover, if $R\in\{T\}'$ and $\{R\}'$ is abelian, 
then $\{T\}'=\{R\}'$ (see {\cite[Proposition 11]{ks14}}, quasianalyticity is not used in the proof here). 
 
Let $T$ be a polynomially bounded operator. Then, clearly, $T$ is power bounded, therefore, $T$ has a unitary asymptote. 
If, in addition,  $T$ is  of class $C_{\cdot 0}$,  then the spectral measure of the unitary asymptote of $T$ is absolutely continuous with respect to $\mathbf{m}$ 
{\cite[Theorem 13 and Proposition 15]{ker4}}. 
If $T$ is quasianalytic, then $T$ is of class $C_{10}$ {\cite[Proposition 33]{ker4}}. 
For definition of the \emph{quasianalytic spectral set} of $T$ we refer to 
\cite{ker4}, \cite{kerntuple} and \cite{ks14}. We recall only that for quasianalytic polynomially bounded operator $T$ 
the quasianalytic spectral set coincides with the measurable  (with respect to $\mathbf{m}$) set on which the spectral measure of  the unitary asymptote  of $T$ is concentrated. 

Let $T$ be a polynomially bounded quasianalytic operator. Then the spectrum $\sigma(T)$ of $T$ is a connected set, and 
$\mathbf{m}(\sigma(T)\cap\mathbb T)>0$. Therefore, if $\sigma(T)\subset\mathbb T$, 
then $\sigma(T)$ is a subarc of $\mathbb T$. Examples of quasianalytic contractions $T$ such that 
$\sigma(T)=\mathbb T$ or $\sigma(T)\cap\mathbb T\neq\mathbb T$ are known. But in all known (to the author) examples
the interior of the polynomially convex hull of $\sigma(T)$ is non-empty. (Recall that 
 the polynomially convex hull of a compact set  $\sigma\subset\mathbb C$ is the union of $\sigma$ and all  the bounded components of 
 $\mathbb C\setminus\sigma$; for example,  the polynomially convex hull of $\mathbb T$ is $\operatorname{clos}\mathbb D$.)
In this paper, a  quasianalytic operator $R$ similar  to a contraction is constructed such that 
$\sigma(R)=\{\mathrm{e}^{\mathrm{i}t}\ :\ t\in[0,\pi]\}$ and the unitary asymptote of $R$ is $U_{\sigma(R)}$.
Therefore, the quasianalytic spectral set of $R$ is $\sigma(R)$. 
For this purpose,  an appropriate quasianalytic operator $T$ with $\sigma(T)=\mathbb T$  is constructed, 
and it is proved that there exists $R\in\{T\}'$ such that $R^2=T$. The existence of non-trivial hyperinvariant subspaces 
of $T$ and $R$ is based on result from \cite{est}. 

The following notation will be used. By $H^p$ the Hardy space is denoted (on a some domain of $\mathbb C$, which will be mentioned). 
By $I_{\mathcal H}$ and $P_{\mathcal E}$ the identity operator on a Hilbert space $\mathcal H$ 
and the orthogonal projection on the subspace $\mathcal E$ are denoted. For two positive functions $w(t)$ and $\phi(t)$, the notation 
$w\asymp \phi$ means that $0<\inf_t w(t)/\phi(t)\leq \sup_t w(t)/\phi(t)<\infty$. By $\mathbf{1}$ and $\chi$ the unit constant function and the identity function are denoted:
 $\mathbf{1}(z)=1$ and $\chi(z)=z$, $z\in\operatorname{clos}\mathbb D$. 

\section{Square root of operator}

\begin{theorem}\label{thmsquare}Let $T$ and $R$ be two operators  on a Hilbert space 
 such that $R^2=T$. Then $R$ is power bounded if and only if $T$ is power bounded, and then
$R$ is of class $C_{1\cdot}$ if and only if $T$ is of class $C_{1\cdot}$, 
and $R$ is of class $C_{\cdot 0}$ if and only if $T$ is of class $C_{\cdot 0}$. 
Furthermore, $R$ is  polynomially bounded if and only if $T$  is polynomially bounded, 
and $R$ is similar to a contraction if and only if $T$  is similar to a contraction.
\end{theorem}

\begin{proof} The proofs of statements concerning power boundedness are  very simple, therefore, they are omitted. 
The proof of ``only if" part of statements concerning polynomial boundedness and similarity of contractions  are  very simple, too. 
Suppose that  $T$  is similar to a contraction. Then $T$ is  completely polynomially bounded \cite{pau}. We will to prove that $R$ 
is  completely polynomially bounded. Then it will be proved that $R$ is similar to a contraction \cite{pau}. 

Let $p$ be a polynomial. Then $p(z)=\sum_{n=0}^Nc_nz^n$, $z\in\mathbb C$, for some $N\in\mathbb N$. 
For convenience, set $c_n=0$ for $n\in\mathbb N$, $n\geq N+1$. Set 
$$ p_0(z)=\sum_{n\geq 0}c_{2n}z^n \ \  \text{ and }  \ \ p_1(z)=\sum_{n\geq 0}c_{2n+1}z^n, \ \ \  z\in\mathbb C.$$

Clearly, 
$$ p_0(z^2)=\frac{p(z)+p(-z)}{2}  \ \  \text{ and } \  \ zp_1(z^2)=\frac{p(z)-p(-z)}{2}, \ \ \  z\in\mathbb C.$$
For a family of polynomials $[p_{ij}]_{i,j=1}^n$ we have 
\begin{align*}[(p_{ij})_0(z^2)]_{i,j=1}^n &=\frac{1}{2}\Bigl([p_{ij}(z)]_{i,j=1}^n+[p_{ij}(-z)]_{i,j=1}^n\Bigr) 
\ \  \text{ and }  \\  
z[(p_{ij})_1(z^2)]_{i,j=1}^n&=\frac{1}{2}\Bigl([p_{ij}(z)]_{i,j=1}^n-[p_{ij}(-z)]_{i,j=1}^n\Bigr),\ \ \  z\in\mathbb C.\end{align*}
Therefore, 
\begin{align*}\|[(p_{ij})_0(z^2)]_{i,j=1}^n\|_{\mathcal L(\ell_n^2)}&\leq
\frac{1}{2}\Bigl(\|[p_{ij}(z)]_{i,j=1}^n\|_{\mathcal L(\ell_n^2)}+\|[p_{ij}(-z)]_{i,j=1}^n\|_{\mathcal L(\ell_n^2)}\Bigr) \\
&\leq \|[p_{ij}]_{i,j=1}^n\|_{H^\infty(\ell_n^2)} \ \  \text{ and }  \\  
|z|\|[(p_{ij})_1(z^2)]_{i,j=1}^n\|_{\mathcal L(\ell_n^2)}&\leq
\frac{1}{2}\Bigl(\|[p_{ij}(z)]_{i,j=1}^n\|_{\mathcal L(\ell_n^2)}+\|[p_{ij}(-z)]_{i,j=1}^n\|_{\mathcal L(\ell_n^2)}\Bigr)\\
&\leq \|[p_{ij}]_{i,j=1}^n\|_{H^\infty(\ell_n^2)},\ \ \  z\in\mathbb C.
\end{align*}
Clearly, for every $\zeta\in\operatorname{clos}\mathbb D$ there exists $z\in\operatorname{clos}\mathbb D$ 
such that $z^2=\zeta$. Therefore, 
\begin{align*}\|[(p_{ij})_0]_{i,j=1}^n\|_{H^\infty(\ell_n^2)} &\leq  \|[p_{ij}]_{i,j=1}^n\|_{H^\infty(\ell_n^2)} 
 \ \  \text{ and }  \\ 
\|[(p_{ij})_1]_{i,j=1}^n\|_{H^\infty(\ell_n^2)} &\leq  \|[p_{ij}]_{i,j=1}^n\|_{H^\infty(\ell_n^2)} .\end{align*}

We have $$[p_{ij}(R)]_{i,j=1}^n= [(p_{ij})_0(T)]_{i,j=1}^n+\Bigl(\oplus_{j=1}^n R\Bigr)\cdot[(p_{ij})_1(T)]_{i,j=1}^n.$$
Since $T$ is  completely polynomially bounded, \eqref{1.2} is fulfilled for $T$ with some constant  $M$. Therefore, 
\begin{align*}\|[p_{ij}(R)]_{i,j=1}^n\|&\leq \|[(p_{ij})_0(T)]_{i,j=1}^n\|+\|R\|\|[(p_{ij})_1(T)]_{i,j=1}^n\|
\\&\leq M\|[(p_{ij})_0]_{i,j=1}^n\|_{H^\infty(\ell_n^2)}+ \|R\| M\|[(p_{ij})_1]_{i,j=1}^n\|_{H^\infty(\ell_n^2)}
\\&=(1+ \|R\|) M\|[p_{ij}]_{i,j=1}^n\|_{H^\infty(\ell_n^2)}.\end{align*}
Thus, $R$ is  completely polynomially bounded.

If we suppose only that $T$ is polynomially bounded, then the proof of polynomial boundedness of $R$ is similar.
 \end{proof}

\begin{lemma}\label{lemsquare} Set $\varrho(\mathrm{e}^{\mathrm{i}t})=\mathrm{e}^{\mathrm{i}\frac{t}{2}}$, $t\in(0,2\pi)$. 
Suppose that $T$, $R\in\mathcal L(\mathcal H)$ are  such that 
$T$ is a polynomially bounded operator of class $C_{\cdot 0}$, $(X, U_{\mathbb T})$ is a unitary asymptote of $T$, 
$R^2=T$ and $XR=\varrho(U_{\mathbb T})X$. Then $(X, \varrho(U_{\mathbb T}))$ is a unitary asymptote of $R$.
\end{lemma}

\begin{proof} By Theorem \ref{thmsquare}, $R$ is a polynomially bounded operator of class $C_{\cdot 0}$. Therefore, 
$R$ has a unitary asymptote $(Y,V)$, and the spectral measure of $V$ is absolutely continuous with respect to $\mathbf{m}$ 
{\cite[Theorem 13 and Proposition 15]{ker4}}. 
By {\cite[Theorem 6]{ker15}}, $(Y,V^2)$ is a unitary asymptote of $T$. 
Therefore, there exists an invertible transformation $Z_1$ such that $Y=Z_1X$ and $Z_1U_{\mathbb T}=V^2Z_1$. 

Since $\varrho(U_{\mathbb T})$ is unitary, there exists a  transformation $Z_2$ 
such that $X=Z_2Y$ and $Z_2V=\varrho(U_{\mathbb T})Z_2$. Therefore, $X=Z_2Z_1X$ and $Z_2Z_1U_{\mathbb T}=U_{\mathbb T}Z_2Z_1$. 
It follows from the definition of a unitary asymptote that 
$$\bigvee_{n\geq 0}U_{\mathbb T}^{-n}X\mathcal H = L^2(\mathbb T).$$
Therefore, $Z_2Z_1=I_{L^2(\mathbb T)}$. Thus, $Z_2=Z_1^{-1}$. 
It is proved that the pairs $(X, \varrho(U_{\mathbb T}))$  and $(Y,V)$ are equivalent. 
\end{proof}

\begin{corollary}\label{corsquare} In assumption of Lemma \ref{lemsquare}, $T$ is quasianalytic if and only if $R$ is quasianalytic.
\end{corollary}
\begin{proof} 
It is well known and easy to see that $\varrho(U_{\mathbb T})$ is unitarily equivalent to $U_\sigma$ with 
$\sigma=\{ \mathrm{e}^{\mathrm{i}t}\ : \ t\in(0,\pi)\}$. Therefore, $\{\varrho(U_{\mathbb T})\}'$ is an abelian algebra. 
Since $\varrho(U_{\mathbb T})\in\{U_{\mathbb T}\}'$ and  $\{U_{\mathbb T}\}'$ is abelian, we conclude that 
 $$\{\varrho(U_{\mathbb T})\}'= \{U_{\mathbb T}\}'$$
by {\cite[Proposition 11]{ks14}} (quasianalyticity is not used in the proof there). 
Thus, 
\begin{equation}\label{hlatsquare}\operatorname{Hlat}\varrho(U_{\mathbb T}) = 
 \operatorname{Hlat}U_{\mathbb T}= \{L^2(\tau)\ : \ \tau\subset\mathbb T\}.\end{equation}

Suppose that $T$ is quasianalytic. It follows from the definition \eqref{defquasi} of quasianalyticity and  \eqref{hlatsquare} 
that  
\begin{equation}\label{ttquasi}X^{-1}L^2(\tau)=\{0\}\end{equation}
 for every measurable set $\tau\subset\mathbb T$ such that $\mathbf{m}(\tau)<1$. 
Since $(X, \varrho(U_{\mathbb T}))$ is a unitary asymptote of $R$, 
we obtain that $R$ is quasianalytic by \eqref{ttquasi} and \eqref{hlatsquare}. 

Conversely, if $R$ is quasianalytic, then $T$ is quasianalytic by the same reasoning. 

Note that ``if" part is a particular case of {\cite[Corollary 13]{ker15}} (applied to $R$). 
\end{proof}

\section{Construction of $T$}

Recall that $\mathbf{1}(z)=1$ and $\chi(z)=z$, $z\in\operatorname{clos}\mathbb D$, and 
 $H^2(\mathbb D)$ denotes the Hardy space in $\mathbb D$.  
Set $S= U_{\mathbb T}|_{H^2(\mathbb D)}$ and    $H^2_-(\mathbb D)=L^2(\mathbb T)\ominus H^2(\mathbb D)$. 

In the following proposition {\cite[Proposition 3.1]{gamstudia}} (based on \cite{cass}), 
{\cite[Theorem 3]{ker1}} (applied to $T^*$) and {\cite[Sec. 5]{ks14}} are combined, therefore, its proof is omitted.

\begin{proposition}\label{propttmain}  Suppose  that $\mathcal H_0$ is a Hilbert space, 
$T_0\in\mathcal L(\mathcal H_0)$ is a contraction of class $C_{00}$, 
 $X_0\in\mathcal L(\mathcal H_0, H^2_-(\mathbb D))$ is such that $\ker X_0=\{0\}$,  
 $\operatorname{clos}X_0\mathcal H_0=H^2_-(\mathbb D)$   and 
$X_0 T_0=P_{H^2_-(\mathbb D)}U_{\mathbb T}|_{H^2_-(\mathbb D)} X_0$.
Put $$T=\begin{pmatrix} S & (\cdot , X_0^\ast \overline\chi)\mathbf{1} \cr \mathbb O & T_0\end{pmatrix}.$$
Then $T$ is similar to a contraction and $T$ is of class $C_{10}$.   Therefore, 
 $T$ admits an $H^\infty$-functional calculus.
Furthermore, $(I_{H^2(\mathbb D)}\oplus X_0, U_{\mathbb T})$ is a unitary asymptote of $T$. Therefore, 
$\mathbb T\subset\sigma(T)\subset\operatorname{clos}\mathbb D$, and $T$ is quasianalytic if and only if 
 \begin{equation}\label{quasidisk}P_{H^2_-(\mathbb D)}L^2(\tau)\cap  X_0\mathcal H_0=\{0\}
\end{equation} for every measurable set $\tau\subset\mathbb T$ such that $\mathbf{m}(\tau)<1$.

Let $\eta\in L^\infty(\mathbb T)$. Then  $\eta\in\widehat\gamma_T(\{T\}')$ (where  $\widehat\gamma_T$ is defined in \eqref{gamma})
 if and only if the mapping 
\begin{equation}\label{gammarr} (I_{H^2(\mathbb D)}\oplus X_0)^{-1}\eta (U_{\mathbb T})(I_{H^2(\mathbb D)}\oplus X_0)
\end{equation} 
is defined and  bounded, and then   $\widehat\gamma_T^{-1}(\eta)$ is equal to the operator in \eqref{gammarr}. 
 \end{proposition}

Let $\nu$ be a positive finite Borel measure on $\mathbb D$. Clearly, the operator of multiplication by the independent variable on $L^2(\nu)$
is a contraction of class $C_{00}$. Denote by 
$P^2(\nu)$ the closure of (analytic) polynomials in $L^2(\nu)$, and by $S_\nu$
the operator of multiplication by the independent variable in $P^2(\nu)$, i.e.
$$S_\nu\in\mathcal L(P^2(\nu)), \ \ \ \ (S_\nu f)(z)=zf(z), \ \ f\in P^2(\nu),
 \ \ z\in \mathbb D.$$
Since $S_\nu$ is the restriction on an invariant subspace of a contraction of class  $C_{00}$,
$S_\nu$ is a contraction of class  $C_{00}$, too. Furthermore, if $H^2(\mathbb D)\subset L^2(\nu)$, then 
 the natural imbedding of $H^2(\mathbb D)$ into $L^2(\nu)$ is bounded and $$P^2(\nu)=\operatorname{clos}_{L^2(\nu)}H^2(\mathbb D).$$  

Set \begin{equation}\label{wwdef} (Wh)(z)=\overline zh(\overline z), \ \ h\in L^2(\mathbb T), \ z\in\mathbb T.\end{equation}
Clearly,  $W\in\mathcal L(L^2(\mathbb T))$  is unitary, $W=W^{-1}$, and $WH^2(\mathbb D)=H^2_-(\mathbb D)$. 

\begin{proposition}\label{propttnu} Let $\nu$   be a positive finite Borel measure on $\mathbb D$ such that  every $f\in P^2(\nu)$ 
is analytic in $\mathbb D$, for every $\lambda\in\mathbb D$ the mapping $f\mapsto f(\lambda)$, $P^2(\nu)\to \mathbb C$ is bounded,  
and $H^2(\mathbb D)\subset P^2(\nu)$. Let $J_\nu\in\mathcal L(H^2(\mathbb D), P^2(\nu))$ be the natural imbedding. 
Set $\mathcal H_0=P^2(\nu)$, $T_0=S_\nu^*$, $X_0=WJ_\nu^*$ and define $T$ as in Proposition \ref{propttmain}. 
Then  $\sigma(T)=\mathbb T$.\end{proposition}

\begin{proof} Let $\lambda \in\mathbb D$. There exists $k_\lambda\in P^2(\nu)$ such that $(f,k_\lambda)=f(\lambda)$
for every $f\in P^2(\nu)$. Clearly, $S_\nu^*k_\lambda = \overline\lambda k_\lambda$. 
 For every $\lambda \in\mathbb D$ and every $f\in P^2(\nu)$ define 
$$f_\lambda(z)=\frac{f(z)-f(\lambda)}{z-\lambda}, \ \ \ z\in\mathbb D.$$ 
Then $f_\lambda\in P^2(\nu)$. Therefore,  $P^2(\nu) = (S_\nu-\lambda I)P^2(\nu) \dotplus \mathbb C\text{{\bf 1}}$ 
for every $\lambda\in\mathbb D$.
(For proof, see,  for example, {\cite[Lemma 4.5]{ars}}.) Thus, $S_\nu-\lambda I$ is left-invertible, therefore, 
$(S_\nu-\lambda I)^*P^2(\nu)=P^2(\nu)$. Furthermore, $$\dim\ker(S_\nu-\lambda I)^*=1.$$ 

Let $\lambda \in\mathbb D$, let $h_\diamond\in H^2(\mathbb D)$, and let $f_\diamond\in P^2(\nu)$. There exists $f\in P^2(\nu)$ such that 
$(S_\nu-\lambda I)^*f=f_\diamond$. Set $h(z)=\frac{h_\diamond(z)-h_\diamond(\overline\lambda)}{z-\overline\lambda}$, $z\in\mathbb D$. Then 
$h\in H^2(\mathbb D)$. Taking into account that 
$$ X_0^\ast \overline\chi=J_\nu W^{-1}\overline\chi=J_\nu\mathbf{1}=\mathbf{1},$$ we obtain that 
$$(T-\overline\lambda I)(h\oplus (f+(h_\diamond(\overline\lambda)-(f,\text{{\bf 1}}))k_\lambda)=h_\diamond\oplus f_\diamond.$$
If $\lambda \in\mathbb D$,  $h\in H^2(\mathbb D)$, and  $f\in P^2(\nu)$ are such that $(T-\overline\lambda I)(h\oplus f)=0$, 
then there exists $c\in\mathbb C$ such that $f=ck_\lambda$ and $(z-\overline\lambda)h(z)=-c$ for every $z\in\mathbb D$. Therefore, $c=0$. 
Thus, $\mathbb D\cap\sigma(T)=\emptyset$. By Proposition \ref{propttmain}, $\mathbb T\subset\sigma(T)\subset\operatorname{clos}\mathbb D$. 
\end{proof}

For $0\leq r<1$ set 
\begin{equation}\label{circle} D_r=\{z\in\mathbb C\ :\ |z-r|<1-r\}, 
 \ \ \Gamma_r=\partial D_r=\{z\in\mathbb C\ :\ |z-r|=1-r\}, 
\end{equation}
denote by $\nu_r$ the arc length measure on $\Gamma_r$. (Of course, $D_0=\mathbb D$, $\Gamma_0=\mathbb T$, and $\nu_0=2\pi\mathbf{m}$.)
 Using a linear change of variable and well-known properties of 
$H^2(\mathbb D)$, it is easily seen that every $f\in P^2(\nu_r)$ is analytic in $D_r$, and for every $\lambda\in D_r$ 
the mapping $f\mapsto f(\lambda)$, $P^2(\nu_r)\to \mathbb C$ is bounded.  

\begin{lemma}\label{lemcircle} Let $\{a_n\}_{n=1}^\infty$ and $\{r_n\}_{n=1}^\infty$ be families of numbers such that $a_n>0$, 
$0<r_{n+1}<r_n$ for every $n$, $\sum_{n=1}^\infty a_n<\infty$, and $r_n\to 0$. Set 
$$ \nu = \frac{1}{2\pi}\sum_{n=1}^\infty a_n \nu_{r_n}.$$
Then $H^2(\mathbb D)\subset L^2(\nu)$, if $f\in P^2(\nu)$, then $f$  
is analytic in $\mathbb D$, and for every $\lambda\in\mathbb D$ the mapping $f\mapsto f(\lambda)$, $P^2(\nu)\to \mathbb C$ is bounded.  
\end{lemma}

\begin{proof} Let $h\in H^2(\mathbb D)$. Since $\frac{1}{2\pi}\int_{\Gamma_r}|h|^2{\mathrm d}\nu_r \leq 2\|h\|_{H^2(\mathbb D)}^2$
(see, for example, {\cite[Lemma I.A.6.3.3]{nik}}), we conclude that $H^2(\mathbb D)\subset L^2(\nu)$. On the other hand, 
$P^2(\nu)\subset P^2(\nu_{r_n})$ for every $n$, $D_{r_n}\subset D_{r_{n+1}}$ and $\cup_{n=1}^\infty D_{r_n}=\mathbb D$.  
Therefore, every $f\in P^2(\nu)$ 
is analytic in $\mathbb D$, and for every $\lambda\in\mathbb D$ the mapping $f\mapsto f(\lambda)$, $P^2(\nu)\to \mathbb C$ is bounded.  
\end{proof}

\begin{remark} The construction of the measure $\nu$ from Lemma \ref{lemcircle} is close to \cite{krietetrutt1}, \cite{krietetrutt2}.
\end{remark}

\section{Transfer to the half-plane and Fourier transform}

Set $\mathbb C_+=\{z\in\mathbb C\ :\ \operatorname{Im}z>0\}$,  $\mathbb C_-=\{z\in\mathbb C\ :\ \operatorname{Im}z<0\}$ and 
\begin{equation}\label{kappa} \varpi(z)=\frac{z-\mathrm{i}}{z+\mathrm{i}}, \ \ \ z\in\mathbb C. 
\end{equation}
It is well known and easy to see that $\varpi|_{\mathbb C_+}$ is a conformal mapping of $\mathbb C_+$ onto $\mathbb D$, and 
for every $0\leq r<1$ and every $f$ for which the integrals below are defined we have 
\begin{equation}\label{jjmunu}\frac{1}{2\pi}\int_{\Gamma_r}f{\mathrm d}\nu_r =
\frac{1}{\pi}\int_{-\infty}^{+\infty}(f\circ\varpi)\Bigl(t+\mathrm{i}\frac{r}{1-r}\Bigr)\frac{{\mathrm d}t}
{t^2+(\frac{1}{1-r})^2},\end{equation}
where $\Gamma_r$ and $\nu_r$ are defined by \eqref{circle} and just after  \eqref{circle}, respectively.
Set \begin{equation}\label{jjkappa}
\mathcal J f(z)=\frac{1}{\sqrt\pi}\frac{1}{z+\mathrm{i}}(f\circ\varpi)(z)
\end{equation} for all functions $f$ and $z\in\mathbb C$ for which the definition \eqref{jjkappa} has sense. 
 Then $\mathcal J$ is a unitary transformation from $L^2(\mathbb T)$ onto $L^2(\mathbb R)$, and  for 
$\eta\in L^\infty(\mathbb T)$ the operator $\mathcal J\eta(U_{\mathbb T})\mathcal J^{-1}$ is 
the multiplication by $\eta\circ\varpi$ acting on $L^2(\mathbb R)$. Furthermore, 
$\mathcal J H^2(\mathbb D)=H^2(\mathbb C_+)$ 
(see, for example, {\cite[Sec. I.A.6.3.1]{nik}}). 
Since $f\in H^2(\mathbb D)$ if and only if
 $f_*(z):= \frac{1}{z}f(\frac{1}{z})$, $|z|>1$, is from  $H^2_-(\mathbb D)$, we have 
$$(\mathcal J f_*)(z)= -(\mathcal J f)(-z), \ \ z\in\mathbb C_-, $$
and
\begin{equation}\label{jjwwjj} (\mathcal JW\mathcal J^{-1}h)(z)=-h(-z) 
\text{ for } h\in H^2(\mathbb C_+) \text{ and } z\in \mathbb C_-.\end{equation}

For a measure $\nu$ defined as in Lemma \ref{lemcircle} set 
\begin{equation}\label{mu} {\mathrm d}\mu = \sum_{n=1}^\infty a_n {\mathrm d}t|_{\mathbb R+\mathrm{i}v_n}
\ \ \text{ with } v_n=\frac{r_n}{1-r_n}, \ \ n\geq 1.
\end{equation}
A straightforward calculation based on \eqref{jjmunu} shows that 
\begin{equation}\label{jjunit}\mathcal J \text{  is a unitary transformation from  } L^2(\mathbb D,\nu) \text{  onto }  L^2(\mathbb C_+,\mu).\end{equation}
Since $P^2(\nu)=\operatorname{clos}_{L^2(\nu)}H^2(\mathbb D)$ and $\mathcal JH^2(\mathbb D)=H^2(\mathbb C_+)$, we conclude that $ H^2(\mathbb C_+)\subset L^2(\mathbb C_+,\mu)$ and 
\begin{equation}\label{munu}\mathcal JP^2(\nu)=\operatorname{clos}_{L^2(\mu)}H^2(\mathbb C_+).\end{equation} 
Denote by $J_\mu$ the natural imbedding of  $ H^2(\mathbb C_+)$ into 
$ \operatorname{clos}_{L^2(\mu)}H^2(\mathbb C_+)$. 

\bigskip

Let $\mathcal D(\mathbb R)$ be the space of test functions, that is, the space of functions from $C^\infty(\mathbb R)$ 
with compact support. Let  $\mathcal S(\mathbb R)$ and $\mathcal S'(\mathbb R)$ be the spaces of rapidly decreasing functions and of tempered distributions, respectively. Recall that  $\mathcal S'(\mathbb R)$ is the dual  space of $\mathcal S(\mathbb R)$, $\mathcal D(\mathbb R)$ is contained and dense in  $\mathcal S(\mathbb R)$, and  $L^p(\mathbb R)\subset\mathcal S'(\mathbb R)$, $1\leq p\leq\infty$. 
The \emph{Fourier transform} $\mathcal F$ of a function $f$ defined on $\mathbb R$ and its inverse $\mathcal F^{-1}$ act by the formulas 
\begin{equation}\label{fourier} (\mathcal F f)(t)=\frac{1}{\sqrt{2\pi}}\int_{\mathbb R}\mathrm{e}^{-\mathrm{i}ts}f(s){\mathrm d}s, \ \ 
(\mathcal F^{-1} f)(t)=\frac{1}{\sqrt{2\pi}}\int_{\mathbb R}\mathrm{e}^{\mathrm{i}ts}f(s){\mathrm d}s,\ \ \ t\in\mathbb R.\end{equation}
It is well known that  $\mathcal F$ and $\mathcal F^{-1}$ are linear continuous mutually inverse bijections 
on $\mathcal S(\mathbb R)$ and on $\mathcal S'(\mathbb R)$,  and 
 $\mathcal F$ is unitary on $L^2(\mathbb R)$. It follows from \eqref{jjwwjj} that 
\begin{equation}\label{ffjjwwjj}\mathcal F\mathcal JW\mathcal J^{-1}= \mathcal JW\mathcal J^{-1}\mathcal F.\end{equation}

For $\Psi\in\mathcal S'(\mathbb R)$ the multiplication $\mathcal M_\Psi$ 
by $\Psi$ and the convolution $\mathcal C_\Psi$ with $\Psi$ are  linear continuous mappings from  $\mathcal S(\mathbb R)$
to $\mathcal S'(\mathbb R)$.  
 If $\Psi\in L^p(\mathbb R)$, $1\leq p\leq\infty$,
then $\mathcal M_\Psi$ and $\mathcal C_\Psi$ act in a usual way:  $(\mathcal M_\Psi f)(t)=\Psi(t) f(t)$ and 
\begin{equation}\label{defconv} (\mathcal C_\Psi f)(t) = \int_{\mathbb R}f(t-s)\Psi(s)\mathrm{d}s, \ \ \ t\in\mathbb R, \ \  \ f\in\mathcal S(\mathbb R).\end{equation}
It is well known that 
\begin{equation}\label{mmcc}\mathcal F\mathcal M_\Psi\mathcal F^{-1}f =
\frac{1}{\sqrt{2\pi}}\mathcal C_{\mathcal F\Psi}f, \ \ \ f\in\mathcal S(\mathbb R).\end{equation}
If $\eta\in L^\infty(\mathbb T)$, then $\mathcal C_{\mathcal F(\eta\circ\varpi)}$ has an extention from 
 $\mathcal S(\mathbb R)$ onto $L^2(\mathbb R)$ defined by \eqref{mmcc},  which is a (linear, bounded) operator on  $L^2(\mathbb R)$ 
 (that is,  $\mathcal C_{\mathcal F(\eta\circ\varpi)}\in\mathcal L(L^2(\mathbb R))$. 
 
 For $\alpha>0$ set 
\begin{equation}\label{thetaplane} \theta_\alpha(z)=\mathrm{e}^{\mathrm{i}\alpha z}, \  \ \ z\in \mathbb C.\end{equation} 
Then\begin{equation}\label{basisalphaf}
\bigl(\mathcal F(\theta_\alpha^nf)\bigr)(t)=(\mathcal F f)(t-n\alpha), \ \ \ t\in\mathbb R, \ \ n\in\mathbb Z, \ \ f\in L^2(\mathbb R)
 \end{equation}
and 
\begin{equation}\label{shiftalpha}\frac{1}{\sqrt{2\pi}}(\mathcal C_{\mathcal F \theta_\alpha}f)(t)=f(t-\alpha), 
\ \ \ t\in\mathbb R, \ \  \ f\in L^2(\mathbb R).\end{equation}
Set $\mathcal K_\alpha=H^2(\mathbb C_+)\ominus\theta_\alpha H^2(\mathbb C_+)$.
By the Paley--Wiener theorem, 
\begin{equation}\label{ffhhplus}H^2(\mathbb C_+)=\mathcal F^{-1}L^2(0,+\infty)\end{equation} and 
 \begin{equation}\label{basisalpha}\theta_\alpha^n\mathcal K_\alpha
=\theta_\alpha^n\mathcal F^{-1}L^2(0,\alpha) = \mathcal F^{-1}L^2(n\alpha,(n+1)\alpha) \text{ for every } n\in\mathbb Z.\end{equation} 
For references see, for example, {\cite[Ch. VI]{katz}} or {\cite[Ch. 7]{rudin}}. 

\bigskip

For $\infty\leq b_1<b_2\leq+\infty$ and $w\colon(b_1,b_2)\to(0,+\infty)$ set
$$L^2((b_1,b_2),w)=\{f\colon(b_1,b_2)\to\mathbb C \ : \ \int_{b_1}^{b_2}|f(t)|^2w(t){\mathrm d}t<\infty\}.$$

\begin{proposition}\label{prophhplus} Let $\nu$ be as in Lemma \ref{lemcircle}, and let $\mu$ be defined by $\nu$ as in \eqref{mu}. 
For $\alpha>0$ set 
\begin{equation}\label{omegan}\frac{1}{\omega^2_\alpha(-n-1)}=\sum_{k=1}^\infty a_k \mathrm{e}^{-2\alpha nv_k},
 \ \ n\geq 0,\end{equation}
\begin{align*}\phi_{\alpha,n}(t)&= \sum_{k=1}^\infty a_k \mathrm{e}^{-2\alpha v_kn}\mathrm{e}^{-2v_kt}, \ \ t\in (0,\alpha),
 \ \ \text{ and } \\  
\phi_\alpha(t)&= \phi_{\alpha,n}(t-n\alpha), \ \ \ t\in (n\alpha,(n+1)\alpha),  \ \ n\geq 0.\end{align*}
Then 
\begin{equation}\label{ffp2mu}\begin{gathered}\mathcal F \text{ is a unitary transformation} \\ \text{  from }
\operatorname{clos}_{L^2(\mu)} H^2(\mathbb C_+) 
\text{ onto }L^2((0,+\infty),\phi_\alpha), \end{gathered}\end{equation}
 \begin{equation}\label{jjmuhh2plus}\begin{aligned}J_\mu^*\operatorname{clos}_{L^2(\mu)}H^2(\mathbb C_+) = 
\{&\oplus_{n=0}^\infty \theta_\alpha^n\mathcal F^{-1}f_n \ :\ f_n\in L^2(0,\alpha), \\ &
\sum_{n=0}^\infty \|f_n\|_{L^2(0,\alpha)}^2\omega^2_\alpha(-n-1)<\infty\},\end{aligned}
\end{equation}
\begin{equation}\label{xx00}\begin{gathered}( \mathcal F\mathcal JW\mathcal J^{-1}J_\mu^*\mathcal F^{-1}f)(t)
=-\phi_\alpha(-t)f(-t), \\ t\in(-\infty,0), \ \ f\in L^2((0,+\infty),\phi_\alpha).\end{gathered}
\end{equation}
\end{proposition}

\begin{proof}
Set $\mathcal L_{\alpha,n}=\theta_\alpha^n\mathcal F^{-1}L^2(0,\alpha)$, $n\geq 0$. By \eqref{basisalpha}, 
$$ H^2(\mathbb C_+)=\oplus_{n=0}^\infty \mathcal L_{\alpha,n} \ \  \text{ and } \  \ 
\|\theta_\alpha^n\mathcal F^{-1}f\|_{H^2(\mathbb C_+)}=\|f\|_{L^2(0,\alpha)}, \ \ \ f\in L^2(0,\alpha).$$

For $v>0$ define  $A_{\alpha,v}\in\mathcal L(L^2(0,\alpha))$ by the formula $(A_{\alpha,v}f)(t)=\mathrm{e}^{-vt}f(t)$, 
$f\in L^2(0,\alpha)$, $t\in(0,\alpha)$.
Then $(\mathcal F^{-1}f)(t+\mathrm{i}v)=(\mathcal F^{-1}A_{\alpha,v}f)(t)$,  $t\in(0,\alpha)$. 

\medskip

Let $n$, $m\geq 0$, and let $f$, $g\in L^2(0,\alpha)$. We have  
\begin{align*}(\theta_\alpha^n&\mathcal F^{-1}f, \theta_\alpha^m\mathcal F^{-1}g)_{L^2(\mu)}
\\ &=
\sum_{k=1}^\infty a_k 
\int_{\mathbb R}(\theta_\alpha^n\mathcal F^{-1}f)(t+\mathrm{i}v_k)\overline{(\theta_\alpha^m\mathcal F^{-1}g)(t+\mathrm{i}v_k)}{\mathrm d}t 
\\&=
\sum_{k=1}^\infty a_k 
\int_{\mathbb R}\mathrm{e}^{\mathrm{i}\alpha(t+\mathrm{i}v_k)n}\overline{\mathrm{e}^{\mathrm{i}\alpha(t+\mathrm{i}v_k)m}}(\mathcal F^{-1}A_{\alpha,v_k}f)(t)\overline{(\mathcal F^{-1}A_{\alpha,v_k}g)(t)}
{\mathrm d}t 
\\&=
\sum_{k=1}^\infty a_k \mathrm{e}^{-\alpha v_k(n+m)}\int_{\mathbb R}\mathrm{e}^{\mathrm{i}(n-m)\alpha t}(\mathcal F^{-1}A_{\alpha,v_k}f)(t)
\overline{(\mathcal F^{-1}A_{\alpha,v_k}g)(t)}
{\mathrm d}t \\&=
\sum_{k=1}^\infty a_k \mathrm{e}^{-\alpha v_k(n+m)}(\theta_\alpha^n\mathcal F^{-1}A_{\alpha,v_k}f,\theta_\alpha^m\mathcal F^{-1}A_{\alpha,v_k}g)_{L^2(\mathbb R)}.
\end{align*}

If  $n\neq m$, then $$(\theta_\alpha^n\mathcal F^{-1}A_{\alpha,v_k}f,\theta_\alpha^m\mathcal F^{-1}A_{\alpha,v_k}g)_{L^2(\mathbb R)}=0,$$ 
because $\mathcal F^{-1}L^2(0,\alpha)=\mathcal K_\alpha$ by \eqref{basisalpha}.
If $n=m$,  then
\begin{equation}\label{normamun}\begin{aligned}
(\theta_\alpha^n&\mathcal F^{-1}f, \theta_\alpha^n\mathcal F^{-1}g)_{L^2(\mu)} 
\\& =\sum_{k=1}^\infty a_k \mathrm{e}^{-2\alpha v_kn}
(\theta_\alpha^n\mathcal F^{-1}A_{\alpha,v_k}f,\theta_\alpha^n\mathcal F^{-1}A_{\alpha,v_k}g)_{L^2(\mathbb R)} 
\\ & =\sum_{k=1}^\infty a_k \mathrm{e}^{-2\alpha v_kn}(A_{\alpha,v_k}f,A_{\alpha,v_k}g)_{L^2(0,\alpha)} 
\\ & = \int_0^\alpha\sum_{k=1}^\infty a_k \mathrm{e}^{-2\alpha v_kn}\mathrm{e}^{-2 v_kt}f(t)\overline{g(t)}\mathrm{d}t
\\ & = \int_0^\alpha f(t)\overline{g(t)}\phi_{\alpha,n}(t)\mathrm{d}t.
\end{aligned}\end{equation}
Clearly, 
\begin{equation}\label{phiomega}
\frac{\mathrm{e}^{-2v_1\alpha}}{\omega^2_\alpha(-n-1)}\leq\frac{1}{\omega^2_\alpha(-n-2)}
 \leq\phi_{\alpha,n}(t)\leq\frac{1}{\omega^2_\alpha(-n-1)}, \ \ t\in(0,\alpha), 
\  n\geq 0.\end{equation}
It is proved that $\mathcal L_{\alpha,n}$ is orthogonal to $\mathcal L_{\alpha,m}$  for $n$, $m\geq 0$ and $n\neq m$, 
and $\mathcal L_{\alpha,n}$ is closed in $L^2(\mu)$ 
for every $n\geq 0$.  
Consequently, 
\begin{equation}\label{closll2mu}\operatorname{clos}_{L^2(\mu)}H^2(\mathbb C_+) = \{\oplus_{n=0}^\infty h_n \ :\ h_n\in \mathcal L_{\alpha,n}, \ \ 
\sum_{n=0}^\infty \| h_n\|_{L^2(\mu)}^2<\infty\}.\end{equation}
Let $\{f_n\}_{n=0}^\infty\subset L^2(0,\alpha)$. Set $h_n=\theta_\alpha^n\mathcal F^{-1}f_n$, $n\geq 0$,  and 
$$f(t)=f_n(t-n\alpha), \ \ \ t\in (n\alpha,(n+1)\alpha), \ \ n\geq 0 .$$
By \eqref{basisalphaf}, $\mathcal F(\oplus_{n=0}^\infty h_n)=f$. The relation \eqref{ffp2mu} follows from the latest equality, 
 \eqref{normamun},  \eqref{closll2mu} and the definition of $\phi_\alpha$. 

Let $J_{\mu,\alpha,n}$ be the natural imbedding of $\mathcal L_{\alpha,n}$ as a subspace of $H^2(\mathbb C_+)$ into  
$\mathcal L_{\alpha,n}$ as a subspace of $L^2(\mu)$. 
Since the spaces $\mathcal L_{\alpha,n}$, $n\geq 0$, are orthogonal and dense in both spaces $H^2(\mathbb C_+)$ and 
$\operatorname{clos}_{L^2(\mu)}H^2(\mathbb C_+)$, we conclude that 
$J_\mu= \oplus_{n=0}^\infty J_{\mu,\alpha,n}$. Consequently, $$J_\mu^*= \oplus_{n=0}^\infty J_{\mu,\alpha,n}^*.$$
Let $f$, $g\in L^2(0,\alpha)$. By  \eqref{normamun},  
\begin{align*}(J_{\mu,\alpha,n}^*\theta_\alpha^n&\mathcal F^{-1}f,\theta_\alpha^n\mathcal F^{-1}g)_{H^2(\mathbb C_+)}= (\theta_\alpha^n\mathcal F^{-1}f, \theta_\alpha^n\mathcal F^{-1}g)_{L^2(\mu)} 
\\ & = \int_0^\alpha f(t)\overline{g(t)}\phi_{\alpha,n}(t)\mathrm{d}t=(\phi_{\alpha,n}f,g)_{L^2(0,\alpha)}
\\ & = (\theta_\alpha^n\mathcal F^{-1}\phi_{\alpha,n}f,\theta_\alpha^n\mathcal F^{-1}g)_{H^2(\mathbb C_+)}.\end{align*}
Thus, \begin{equation}\label{jjstar}
J_{\mu,\alpha,n}^*\theta_\alpha^n\mathcal F^{-1}f = \theta_\alpha^n\mathcal F^{-1}\phi_{\alpha,n}f, \ \ \ f\in L^2(0,\alpha).
\end{equation}
By \eqref{jjstar} and \eqref{closll2mu},
\begin{align*}J_\mu^*\operatorname{clos}_{L^2(\mu)}H^2(\mathbb C_+) = 
\{\oplus_{n=0}^\infty \theta_\alpha^n\mathcal F^{-1}&\phi_{\alpha,n}f_n \ :\ f_n\in L^2(0,\alpha), \\&
 \sum_{n=0}^\infty \| \theta_\alpha^n\mathcal F^{-1}f_n\|_{L^2(\mu)}^2<\infty\}.\end{align*}
Set $g_n=\phi_{\alpha,n}f_n$, then 
\begin{align*}J_\mu^*\operatorname{clos}_{L^2(\mu)}H^2(\mathbb C_+) = 
\{\oplus_{n=0}^\infty \theta_\alpha^n\mathcal F^{-1}&g_n \ :\ g_n\in L^2(0,\alpha), 
\\& \sum_{n=0}^\infty\Bigl\| \theta_\alpha^n\mathcal F^{-1}\frac{g_n}{\phi_{\alpha,n}}\Bigr\|_{L^2(\mu)}^2<\infty\}.\end{align*}
By \eqref{normamun}, 
$$\Bigl\| \theta_\alpha^n\mathcal F^{-1}\frac{g_n}{\phi_{\alpha,n}}\Bigr\|_{L^2(\mu)}^2  
 = \int_0^\alpha \frac{|g_n(t)|^2}{\phi_{\alpha,n}(t)^2}\phi_{\alpha,n}(t){\mathrm d}t=
\int_0^\alpha \frac{|g_n(t)|^2}{\phi_{\alpha,n}(t)}{\mathrm d}t. $$
It follows from the latest equality and  \eqref{phiomega} that
 $$\sum_{n=0}^\infty \Bigl\| \theta_\alpha^n\mathcal F^{-1}\frac{g_n}{\phi_{\alpha,n}}\Bigr\|_{L^2(\mu)}^2<\infty 
\ \text{ if and only if } \ \sum_{n=0}^\infty \|g_n\|_{L^2(0,\alpha)}^2\omega^2_\alpha(-n-1)<\infty. $$
The equality \eqref{jjmuhh2plus} is proved. 

 Let $f\in L^2((0,+\infty),\phi_\alpha)$. Then $f=\oplus _{n=0}^\infty f|_{(n\alpha,(n+1)\alpha)}$. 
Set $$f_n(t)=f(t+n\alpha),  \ \ \ t\in(0,\alpha), \ \ n\geq 0.$$ 
  By  \eqref{basisalphaf} and \eqref{jjstar}, 
$$J_\mu^*\mathcal F^{-1}\!f|_{(n\alpha,(n+1)\alpha)}\! = J_\mu^*\theta_\alpha^n\mathcal F^{-1}\!f_n  
=\theta_\alpha^n\mathcal F^{-1}(\phi_{\alpha,n}f_n)=\mathcal F^{-1}\!\bigl(\!(\phi_\alpha f)|_{(n\alpha,(n+1)\alpha)}\!\bigr).$$
The equality \eqref{xx00} follows from \eqref{ffjjwwjj} and \eqref{jjwwjj}.
\end{proof}

\begin{remark} The idea of  Proposition \ref{prophhplus} is from \cite{fr1}, \cite{frankfurtrovnyak}.\end{remark}

\begin{theorem}\label{thmconv} Let $T$ be defined as in Proposition \ref{propttnu} with $\nu$ as in Lemma \ref{lemcircle}. 
Define  $\mu$  as in \eqref{mu}. For $\alpha>0$ set 
$$\widetilde\phi_\alpha\colon\!\mathbb R\!\to\!(0,+\infty), \ \  \ \widetilde\phi_\alpha(t)=\frac{1}{\phi_\alpha(-t)}, \ \ t\!\in\!(-\infty,0), 
\ \ \  \widetilde\phi_\alpha(t)=1,\ \ t\!\in\!(0,+\infty),$$
where $\phi_\alpha$ is defined as in Proposition \ref{prophhplus}.

 Let $\eta\in L^\infty$.  
Then $\eta\in\widehat\gamma_T(\{T\}')$ (where  $\widehat\gamma_T$ is defined in \eqref{gamma})
 if and only if  
$\mathcal C_{\mathcal F(\eta\circ\varpi)}\in\mathcal L(L^2(\mathbb R,\widetilde\phi_\alpha))$, 
 and then $\widehat\gamma_T^{-1}(\eta)$  is unitarily equivalent to $\frac{1}{\sqrt{2\pi}}\mathcal C_{\mathcal F(\eta\circ\varpi)}\in\mathcal L(L^2(\mathbb R,\widetilde\phi_\alpha))$.
\end{theorem}

\begin{proof} By Proposition \ref{propttmain},  $\eta\in\widehat\gamma_T(\{T\}')$ 
 if and only if the mapping from \eqref{gammarr} is defined and  bounded.  By \eqref{jjunit}, \eqref{munu}, and \eqref{ffp2mu},   
$$ \mathcal F\mathcal J\colon H^2(\mathbb D)\oplus P^2(\nu)\to L^2(0,+\infty)\oplus L^2((0,+\infty),\phi_\alpha)\ \text{is unitary}.$$
 Set $$Y_\alpha=  (\mathcal F\mathcal J)_{H^2_-(\mathbb D)\to L^2(-\infty,0)}X_0
(\mathcal F\mathcal J)_{ L^2((0,+\infty),\phi_\alpha)\to P^2(\nu)}^{-1},$$
where lower indesis of $\mathcal F\mathcal J$ and $(\mathcal F\mathcal J)^{-1}$ show spaces between they act. 
Taking into account the definition of $X_0$, the equality
\begin{equation}\label{jjmunustar}J_\nu^*=\mathcal J^{-1}J_\mu^*\mathcal J,
\end{equation}
 applying  equalities \eqref{xx00} and \eqref{munu}, 
we obtain that $Y_\alpha$ acts by the formula 
\begin{align*}Y_\alpha\colon   L^2((0,+\infty),\phi_\alpha) \to  L^2(-\infty,0), 
\ \ & (Y_\alpha f)(t) = -\phi_\alpha(-t)f(-t), \\& 
t\in(-\infty,0), \ \ f\in  L^2((0,+\infty),\phi_\alpha).\end{align*}
By \eqref{mmcc},  the mapping from \eqref{gammarr} is defined and bounded if and only if the mapping  
\begin{equation}\label{gammarr1}(I_{ L^2(0,+\infty)}\oplus Y_\alpha^{-1})\frac{1}{\sqrt{2\pi}}\mathcal C_{\mathcal F(\eta\circ\varpi)}
(I_{ L^2(0,+\infty)}\oplus Y_\alpha)\end{equation} is defined and bounded.  

Define $V_\alpha\colon   L^2((-\infty,0),\widetilde\phi_\alpha) \to \colon   L^2((0,+\infty),\phi_\alpha)$ by the formula 
$$(V_\alpha f)(t) = -\frac{1}{\phi_\alpha(t)}f(-t), \ \ \ t\in(0,+\infty), \ \ f \in   L^2((-\infty,0),\widetilde\phi_\alpha).$$
Then $V_\alpha$ is unitary and $Y_\alpha V_\alpha$ is the natural imbedding of  $ L^2((-\infty,0),\widetilde\phi_\alpha) $ 
into $ L^2(-\infty,0)$. Thus, $I_{ L^2(0,+\infty)}\oplus Y_\alpha V_\alpha$ is the natural imbedding of  
$ L^2(\mathbb R, \widetilde\phi_\alpha)$ into $ L^2(\mathbb R)$. 
Multiplying the mapping from \eqref{gammarr1} by $I_{ L^2(0,+\infty)}\oplus V_\alpha^{-1}$ 
from the left side and by $I_{ L^2(0,+\infty)}\oplus V_\alpha$ from the right side, we obtain the conclusion of the theorem.
\end{proof}

\section{Properties of constructed weights}

Let $w,\phi\colon \mathbb R\to(0,+\infty)$ be two measurable functions. If $w\asymp \phi$, then the natural imbedding 
$$ L^2(\mathbb R,w)\to L^2(\mathbb R,\phi)$$ is a (bounded) transformation with the bounded inverse. Let $\alpha>0$. 
Let $\widetilde\phi_\alpha $ be defined in Theorem \ref{thmconv}, and 
let ${\omega^2_\alpha(-n-1)}$, $n\geq 0$, be defined by \eqref{omegan}. 
 We may assume that $\sum_{k=1}^\infty a_k\leq 1$.  Then  $\omega^2_\alpha(-n-1)\geq 1$ 
for $n\geq 0$. 
Set 
\begin{equation}\label{omegarr} \begin{aligned} w_\alpha(t) & = \omega^2_\alpha(-n-1), \ \ \ t\in(-(n+1)\alpha,-n\alpha),
 \ \ n\geq 0, \\  w_\alpha(t) & =1, \ \ t\in(0,+\infty).\end{aligned}\end{equation}
By \eqref{phiomega}, 
\begin{equation}\label{wphi}w_\alpha\asymp \widetilde\phi_\alpha. 
\end{equation}
 Therefore, we can consider $ L^2(\mathbb R,w_\alpha)$ instead of $L^2(\mathbb R, \widetilde\phi_\alpha )$.

\begin{lemma}\label{lemsubmult}Let $\{a_k\}_{k=1}^\infty$ and $\{v_k\}_{k=1}^\infty$ be families of numbers such that $0<a_k\leq 1$, 
$0<v_{k+1}<v_k$ for every $k$,  and $\sum_{k=1}^\infty a_k<\infty$. For $\alpha>0$ define 
${\omega^2_\alpha(-n-1)}$, $n\geq 0$, by \eqref{omegan}.
Then 
$$ \omega^2_\alpha(-n-m-1)\leq\Bigl(1+2\sum_{k=1}^\infty a_k\Bigr)\omega^2_\alpha(-n-1)
\omega^2_\alpha(-m-1), \ \ \ m,n\geq 0.$$
\end{lemma}

\begin{proof} We have 
\begin{align*}&\frac{1}{\omega^2_\alpha(-n-1)\omega^2_\alpha(-m-1)}=
\sum_{k=1}^\infty a_k \mathrm{e}^{-2\alpha nv_k}\sum_{l=1}^\infty a_l \mathrm{e}^{-2\alpha mv_l}
\\ &=
\sum_{k=1}^\infty a_k \mathrm{e}^{-2\alpha nv_k}\Bigl(\sum_{l=1}^{k-1} a_l \mathrm{e}^{-2\alpha mv_l}+ 
a_k \mathrm{e}^{-2\alpha mv_k}+\sum_{l=k+1}^\infty a_l \mathrm{e}^{-2\alpha mv_l}\Bigr)
\\ &=
\sum_{k=2}^\infty a_k \mathrm{e}^{-2\alpha nv_k}\sum_{l=1}^{k-1} a_l \mathrm{e}^{-2\alpha mv_l}+
\sum_{k=1}^\infty a_k^2 \mathrm{e}^{-2\alpha (n+m)v_k} 
\\ &\quad\quad\quad \quad\quad\quad\quad\quad\quad  + \sum_{k=2}^\infty a_k \mathrm{e}^{-2\alpha mv_k}\sum_{l=1}^{k-1} a_l \mathrm{e}^{-2\alpha nv_l}
\\&
\leq \sum_{k=2}^\infty a_k \mathrm{e}^{-2\alpha nv_k} \mathrm{e}^{-2\alpha mv_k}\sum_{l=1}^{k-1} a_l +
\sum_{k=1}^\infty a_k \mathrm{e}^{-2\alpha (n+m)v_k} 
\\ &\quad\quad\quad\quad\quad\quad \quad\quad\quad + \sum_{k=2}^\infty a_k \mathrm{e}^{-2\alpha mv_k}\mathrm{e}^{-2\alpha nv_k}\sum_{l=1}^{k-1} a_l 
\\&
\leq 2\sum_{l=1}^\infty a_l\sum_{k=2}^\infty a_k \mathrm{e}^{-2\alpha (n+m)v_k}+ \sum_{k=1}^\infty a_k \mathrm{e}^{-2\alpha (n+m)v_k}
\\&=\Bigl(1+2\sum_{k=1}^\infty a_k\Bigr)\frac{1}{\omega^2_\alpha(-n-m-1)}.\qedhere
\end{align*}
\end{proof}

\begin{corollary}\label{corsubmult}Let $\{a_k\}_{k=1}^\infty$ and $\{v_k\}_{k=1}^\infty$ be families of numbers such that $a_k>0$, 
$0<v_{k+1}<v_k$ for every $k$,  and $\sum_{k=1}^\infty a_k\leq 1$.  Set $$C=1+2\sum_{k=1}^\infty a_k.$$ For $\alpha>0$  
define $w_\alpha$ by \eqref{omegarr}.  Then 
$$ w_\alpha(t+s)\leq C^2\omega_\alpha(-2)^ 2 w_\alpha(t)w_\alpha(s), \ \ \ t,s\in\mathbb R.$$
\end{corollary}

\begin{proof} First, consider the case where $t,s<0$. Then there exists $m,n\geq 0$ such that $t\in(-(n+1)\alpha,-n\alpha)$ and 
$s\in(-(m+1)\alpha,-m\alpha)$. Then $t+s\in(-(n+m+2)\alpha,-(n+m)\alpha)$. Therefore, 
\begin{align*}w_\alpha(t+s)\leq \omega^2_\alpha(-n-m-2) & \leq C\omega^2_\alpha(-n-1)\omega^2_\alpha(-m-2) \\ 
\leq 
 C\omega^2_\alpha(-n-1)C\omega^2_\alpha(-m-1)\omega^2_\alpha(-2) &  =
 C^2\omega^2_\alpha(-2) w_\alpha(t)w_\alpha(s).\end{align*}
In remaining cases the conclusion of the lemma follows from the fact that $w_\alpha$ is nonincreasing and 
 $w_\alpha\equiv1$ on $(0,+\infty)$. 
\end{proof}

\begin{lemma}\label{lemexp} Let $\{a_k\}_{k=1}^\infty$ and $\{v_k\}_{k=1}^\infty$ be families of numbers such that $a_k>0$, 
$0<v_{k+1}<v_k$ for every $k$, $v_k\to 0$ and $\sum_{k=1}^\infty a_k<\infty$. 
For $\alpha>0$  define $\omega^2_\alpha(-n-1)$, $n\geq 0$, by \eqref{omegan}. 
Then for every $\varepsilon>0$ there exists a finite constant $C_\varepsilon$ (which also depends on $\alpha$) such that 
$\omega^2_\alpha(-n-1)\leq C_\varepsilon \mathrm{e}^{\varepsilon n}$ for all $n\geq 0$.
\end{lemma}

\begin{proof} Since $v_k\to 0$, there exists $k_\varepsilon$ such that $2\alpha v_{k_\varepsilon}\leq\varepsilon$. 
We have $$\frac{1}{\omega^2_\alpha(-n-1)}=\sum_{k=1}^\infty a_k \mathrm{e}^{-2\alpha nv_k}\geq
 \sum_{k=k_\varepsilon}^\infty a_k \mathrm{e}^{-2\alpha nv_k}\geq\sum_{k=k_\varepsilon}^\infty a_k \mathrm{e}^{-\varepsilon n}.$$
Thus, $ C_\varepsilon =\frac{1}{\sum_{k=k_\varepsilon}^\infty a_k}.$
 \end{proof}

Recall that $\mathcal C_\varphi$ denote the convolution with a function $\varphi$, see \eqref{defconv}. 

\begin{lemma}\label{lemconvsubmult} Let $C>0$, and let $w\colon \mathbb R\to(0,+\infty)$  and $\psi\colon \mathbb R\to\mathbb C$ be  measurable
functions such that 
$w(t+s)\leq C w(t)w(s)$ for all $s,t\in\mathbb R$ and  
$\psi \sqrt w\in L^1(\mathbb R)$. 
Then $\mathcal C_\psi\in\mathcal L(L^2(\mathbb R, w))$ and $\|\mathcal C_\psi\|\leq \sqrt C\|\psi \sqrt w\|_{L^1(\mathbb R)}.$
\end{lemma}

\begin{proof} Define $ B\in\mathcal L(L^2(\mathbb R, w), L^2(\mathbb R))$ by the formula $Bf=\sqrt w f$, $f\in L^2(\mathbb R, w)$. 
Then $B$ is unitary, and $$(B \mathcal C_\psi B^{-1} f)(t)=\int_{\mathbb R}\psi(t-s)f(s)\frac{\sqrt {w(t)}}{\sqrt {w(s)}}{\mathrm d}s,
 \ \ t\in\mathbb R, \ \ f\in L^2(\mathbb R).$$
Therefore,
\begin{align*}|(B \mathcal C_\psi B^{-1} f)(t)|&\leq\int_{\mathbb R}|\psi(t-s)f(s)|\frac{\sqrt {w(t)}}{\sqrt {w(s)}}
{\mathrm d}s \\& \leq
\int_{\mathbb R}|\psi(t-s)f(s)|\frac{\sqrt{C w(t-s)w(s)}}{\sqrt {w(s)}}{\mathrm d}s\\ &= 
\sqrt C\int_{\mathbb R}|\psi(t-s)|\sqrt{w(t-s)}|f(s)|{\mathrm d}s \\& =\sqrt C\bigl(\mathcal C_{|\psi|\sqrt w} |f|\bigr)(t), 
\ \ t\in\mathbb R, \ \ f\in L^2(\mathbb R).
\end{align*}
It is well known (and can be deduced from \eqref{mmcc}) that if  $\varphi\in L^1(\mathbb R)$, then  
$\mathcal C_\varphi\in\mathcal L(L^2(\mathbb R))$ and 
$\|\mathcal C_\varphi\|\leq \|\varphi\|_{L^1(\mathbb R)}$. Setting $\varphi= |\psi|\sqrt w$, we obtain 
$$ \|B \mathcal C_\psi B^{-1} f\|_{L^2(\mathbb R)}\leq\sqrt C\||\psi|\sqrt w\|_{L^1(\mathbb R)}\| |f|\|_{L^2(\mathbb R)}.$$
Clearly, $\||\psi|\sqrt w\|_{L^1(\mathbb R)} = \|\psi\sqrt w\|_{L^1(\mathbb R)}$ and $\| |f|\|_{L^2(\mathbb R)}=\| f\|_{L^2(\mathbb R)}.$
The conclusion of the lemma follows from the  unitarity of $B$. 
 \end{proof}

Recall that $\varpi$ is defined by \eqref{kappa} and $\mathcal F$ is the Fourier transform (see \eqref{fourier}). 

\begin{lemma}\label{lemlogconv} Let $C>0$, and let $w\colon \mathbb R\to[1,+\infty)$   be a nonincreasing 
function such that $w(t+s)\leq C w(t)w(s)$ for all $s,t\in\mathbb R$ and 
\begin{equation}\label{wlogfinit}\int_{\mathbb R}\frac{\log w(t)}{1+t^2}<\infty.\end{equation}
Then there exists $\eta\in L^\infty(\mathbb T)$ such that $\eta(\mathrm{e}^{\mathrm{i}t})=0$ for $t\in(\pi,2\pi)$, 
$\eta\not\equiv 0$,  and 
$$\mathcal C_{\mathcal F(\eta\circ\varpi)}\in\mathcal L(L^2(\mathbb R,w)).$$ 
\end{lemma}
\begin{proof}
There exists a function $h\colon \mathbb R\to (0,+\infty)$ such that $h\in L^1(\mathbb R)\cap L^2(\mathbb R)$, 
$h\sqrt w\in L^1(\mathbb R)$ and 
\begin{equation}\label{hlogfinit}\int_{\mathbb R}\frac{\log h(t)}{1+t^2}>-\infty.\end{equation}
Indeed, take $c>0$, $\varepsilon>0$ and set 
$$h(t)=1, \ \ t\in(-c,c) \ \ \text{ and }\ \ h(t)=\frac{1}{|t|^{1+\varepsilon}\sqrt{w(t)}}, \ \  t\in(\infty,-c)\cup(c,+\infty).$$
Since $w$ is nonincreasing, $w$ is bounded on $(-c,c)$, therefore, $h\sqrt w\in L^1(\mathbb R)$. Since $w\geq 1$ on $\mathbb R$, 
$h\in L^1(\mathbb R)\cap L^2(\mathbb R)$. Furthermore,
 $$\int_{\mathbb R}\frac{\log h(t)}{1+t^2}=\Bigl(\int_{-\infty}^{-c}+\int_c^{+\infty}\Bigr)
 \Bigl(-\frac{(1+\varepsilon)\log |t|+\frac{1}{2}\log w(t)}{1+t^2}\Bigr)>-\infty$$ by \eqref{wlogfinit}.

By \eqref{hlogfinit}, there exists $\psi\in H^1(\mathbb C_+)\cap H^2(\mathbb C_+)$ such that $|\psi|=h$ a.e. on $\mathbb R$
(see, for example, {\cite[Theorem II.4.4]{garnett}}). By Lemma \ref{lemconvsubmult}, $\mathcal C_\psi\in\mathcal L(L^2(\mathbb R,w))$. 
Set $\eta=(\mathcal F^{-1}\psi)\circ\varpi^{-1}$. Since  $\psi\in H^2(\mathbb C_+)$ and 
$(\mathcal F^{-1}\psi)(-t)=(\mathcal F\psi)(t)$ for $t\in\mathbb R$, we have 
$(\mathcal F^{-1}\psi)\in L^2(-\infty,0)$ by \eqref{ffhhplus}. It remains to note that 
$\varpi\bigl((0,+\infty)\bigr)=\{\mathrm{e}^{\mathrm{i}t}\ :\ t\in(\pi,2\pi)\}$.
 \end{proof}

The following lemma will be applied in Sec. 9. 

\begin{lemma}\label{lemsupport}Suppose that $\delta>0$, and  $B\in\mathcal L(L^2(\mathbb R))$ is 
such that for every $-\infty<b_1<b_2<+\infty$  $$ B L^2(b_1,b_2)\subset L^2(b_1-\delta,b_2+\delta).$$
Suppose that a sequence $\{\omega(n)\}_{n\in\mathbb Z}$ of positive numbers is such that 
$$\frac{\omega(n)}{\omega(n+1)}\asymp 1, \ \ n\in\mathbb Z.$$
 For $\alpha>\delta$ set $w_\alpha(t)=\omega^2(n)$, 
$t\in(n\alpha,(n+1)\alpha)$, $n\in\mathbb Z$.  Then $B\in\mathcal L(L^2(\mathbb R, w_\alpha))$, and 
$$\|B\|_{\mathcal L(L^2(\mathbb R, w_\alpha))}^2\leq 3\|B\|_{\mathcal L(L^2(\mathbb R))}^2
\Bigl(1+\sup_{n\in\mathbb Z}\frac{\omega^2(n)}{\omega^2(n+1)} + 
\sup_{n\in\mathbb Z}\frac{\omega^2(n)}{\omega^2(n-1)}\Bigr) .$$
\end{lemma}
\begin{proof}
Let $f\in L^2(\mathbb R, w_\alpha)$. Set $f_n=f|_{(n\alpha,(n+1)\alpha)}$, $n\in\mathbb Z$.  By assumption, 
$Bf_n$ is well-defined, and $Bf_n\in L^2((n-1)\alpha,(n+2)\alpha)$. 
We have 
\begin{align*}Bf &= \sum_{n\in\mathbb Z}(Bf_n|_{((n-1)\alpha,n\alpha)}+Bf_n|_{(n\alpha,(n+1)\alpha)}+
Bf_n|_{((n+1)\alpha, (n+2)\alpha)}) 
\\ & =
\oplus_{n\in\mathbb Z}(Bf_{n-1}+Bf_n + Bf_{n+1})|_{(n\alpha,(n+1)\alpha)}.\end{align*}
Furthermore,
\begin{align*}&\|(Bf_{n-1}+Bf_n + Bf_{n+1})|_{(n\alpha,(n+1)\alpha)}\|_{L^2(\mathbb R, w_\alpha)}^2 
\\ &=\omega^2(n)\|(Bf_{n-1}+Bf_n + Bf_{n+1})|_{(n\alpha,(n+1)\alpha)}\|_{L^2(\mathbb R)}^2 
\\ &\leq 
3\omega^2(n)\bigl(\|(Bf_{n-1})|_{(n\alpha,(n+1)\alpha)}\|_{L^2(\mathbb R)}^2+
\|(Bf_n)|_{(n\alpha,(n+1)\alpha)}\|_{L^2(\mathbb R)}^2
\\&\quad\quad\quad\quad\quad\quad \quad\quad\quad\quad\quad\quad\quad\quad +\|(Bf_{n+1})|_{(n\alpha,(n+1)\alpha)}\|_{L^2(\mathbb R)}^2\bigr)
\\ &\leq3\omega^2(n)\|B\|_{\mathcal L(L^2(\mathbb R))}^2\bigl(\|f_{n-1}\|_{L^2(\mathbb R)}^2 +
\|f_n\|_{L^2(\mathbb R)}^2+\|f_{n+1}\|_{L^2(\mathbb R)}^2\bigr) 
 \\ &= 
3\omega^2(n)\|B\|_{\mathcal L(L^2(\mathbb R))}^2\Bigl(\frac{1}{\omega^2(n-1)}\|f_{n-1}\|_{L^2(\mathbb R, w_\alpha)}^2  +
\frac{1}{\omega^2(n)}\|f_n\|_{L^2(\mathbb R, w_\alpha)}^2 
\\&\quad\quad\quad\quad\quad\quad \quad\quad\quad\quad\quad\quad\quad\quad +\frac{1}{\omega^2(n+1)}\|f_{n+1}\|_{L^2(\mathbb R, w_\alpha)}^2\Bigr)
\\ &=
3\|B\|_{\mathcal L(L^2(\mathbb R))}^2\Bigl(\frac{\omega^2(n)}{\omega^2(n-1)}\|f_{n-1}\|_{L^2(\mathbb R, w_\alpha)}^2  +
\|f_n\|_{L^2(\mathbb R, w_\alpha)}^2  
\\&\quad\quad\quad\quad\quad\quad \quad\quad\quad\quad\quad\quad\quad\quad+\frac{\omega^2(n)}{\omega^2(n+1)}\|f_{n+1}\|_{L^2(\mathbb R, w_\alpha)}^2\Bigr).\end{align*}
Therefore, \begin{align*}&\|Bf\|_{L^2(\mathbb R, w_\alpha)}^2
=\sum_{n\in\mathbb Z}\|(Bf_{n-1}+Bf_n + Bf_{n+1})|_{(n\alpha,(n+1)\alpha)}\|_{L^2(\mathbb R, w_\alpha)}^2
\\ &\leq 3\|B\|_{\mathcal L(L^2(\mathbb R))}^2 \sum_{n\in\mathbb Z}\Bigl(\frac{\omega^2(n)}{\omega^2(n-1)}\|f_{n-1}\|_{L^2(\mathbb R, w_\alpha)}^2  +
\|f_n\|_{L^2(\mathbb R, w_\alpha)}^2 
\\&\quad\quad\quad\quad\quad\quad \quad\quad\quad\quad\quad\quad\quad\quad+\frac{\omega^2(n)}{\omega^2(n+1)}\|f_{n+1}\|_{L^2(\mathbb R, w_\alpha)}^2\Bigr)
\\& =
 3\|B\|_{\mathcal L(L^2(\mathbb R))}^2 \sum_{n\in\mathbb Z}\Bigl(\frac{\omega^2(n+1)}{\omega^2(n)} +1   +
\frac{\omega^2(n-1)}{\omega^2(n)}\Bigr)\|f_{n}\|_{L^2(\mathbb R, w_\alpha)}^2
\\ &\leq
3\|B\|_{\mathcal L(L^2(\mathbb R))}^2\Bigl(1+\sup_{n\in\mathbb Z}\frac{\omega^2(n)}{\omega^2(n+1)} + 
\sup_{n\in\mathbb Z}\frac{\omega^2(n)}{\omega^2(n-1)}\Bigr)
\sum_{n\in\mathbb Z}\|f_{n}\|_{L^2(\mathbb R, w_\alpha)}^2
\\& =3\|B\|_{\mathcal L(L^2(\mathbb R))}^2\Bigl(1+\sup_{n\in\mathbb Z}\frac{\omega^2(n)}{\omega^2(n+1)} + 
\sup_{n\in\mathbb Z}\frac{\omega^2(n)}{\omega^2(n-1)}\Bigr)\|f\|_{L^2(\mathbb R, w_\alpha)}^2 \qedhere
\end{align*}
\end{proof}

\section{Quasianalyticity}

We will apply Beurling's quasianalyticity theorem, see {\cite[Ch. VII.B.5]{koosis}}.  

\begin{theorem}[Beurling]\label{thmbeurling}
Let $-\infty<b_1<b_2<+\infty$. For $c>0$ set 
$$\mathcal G(c)=\{t+\mathrm{i}y\ :\ b_1<t<b_2, \ 0<y<c\}.$$ For a function $f$ analytic in $\mathcal G(c)$ set 
$$\varsigma_{\mathcal G(c)}(f)=\sup_{0<y<c}\Bigl(\int_{b_1}^{b_2}|f(t+\mathrm{i}y)|^2{\mathrm d}t\Bigr)^{1/2}.$$
For 
$\varphi\in L^2(b_1,b_2)$ and 
$u\in [1,+\infty)$ define $M(u)$ by the relation
\begin{align*} \mathrm{e}^{-M(u)}=\inf\Bigl\{\Bigl(\int_{b_1}^{b_2}|\varphi(t)-f(t)|^2{\mathrm d}t\Bigr)^{1/2} \ : 
&\ f \text{ is analytic in }\mathcal G(c) \\&
\text{ and }\  \varsigma_{\mathcal G(c)}(f)\leq \mathrm{e}^u\Bigr\}.\end{align*}
If $$\int_{1}^{+\infty} \frac{M(u)}{u^2}{\mathrm d}u=\infty $$ and 
$|\{t\in \mathbb R\ :\ b_1<t<b_2, \ \varphi(t)=0\}|>0$ 
(where $|\cdot|$ is the linear measure of a subset of $\mathbb R$), then $\varphi\equiv 0$.
\end{theorem}

To apply Theorem \ref{thmbeurling} we need the following simple lemma.   

\begin{lemma}\label{leminfty} Let $M\colon (0,+\infty)\to (0,+\infty)$ be a nondecreasing function. Let $\alpha>0$. Then 
 $$\sum_{n=1}^\infty \frac{M(n\alpha )}{n^2}=\infty \ \ \text{ if and only if } \ \  \int_{1}^{+\infty} \frac{M(u)}{u^2}{\mathrm d}u=\infty .$$
\end{lemma} 

\begin{proof} We have $$\int_1^{+\infty} \frac{M(u\alpha)}{u^2}{\mathrm d}u=\sum_{n=1}^\infty\int_{n}^{n+1} \frac{M(u\alpha)}{u^2}{\mathrm d}u$$ and
$$\frac{M(n\alpha)}{(n+1)^2}\leq \frac{M(u\alpha)}{u^2} \leq \frac{M((n+1)\alpha)}{n^2} \ \ \text{ for } u\in[n,n+1].$$ 
Since $\frac{n}{n+1}\to 1$ when $n\to \infty$, we conclude that 
$$\sum_{n=1}^\infty \frac{M(n\alpha)}{n^2}=\infty \ \ \text{ if and only if } \ \  \int_1^{+\infty} \frac{M(u\alpha)}{u^2}{\mathrm d}u=\infty .$$
The lemma follows from the equality $$ \int_{1}^{+\infty} \frac{M(u\alpha)}{u^2}{\mathrm d}u=
\alpha \int_{\alpha}^{+\infty} \frac{M(u)}{u^2}{\mathrm d}u.    \qedhere$$
\end{proof}

\bigskip

\begin{theorem}\label{thmquasi} Let $T$ be defined as in Proposition \ref{propttnu} with $\nu$ as in Lemma \ref{lemcircle}. 
 For $\alpha>0$   define $\omega^2_\alpha(-n-1)$, $n\geq 0$, by \eqref{omegan}. 
Then $T$ is quasianalytic if and only if 
\begin{equation}\label{quasi}\sum_{n=0}^\infty \frac{\log \omega_\alpha(-n-1)}{(n+1)^2}=\infty.\end{equation}
 \end{theorem} 

\begin{proof}  \emph{``If" part.}  By Proposition \ref{propttmain}, $T$ is quasianalytic if and only if \eqref{quasidisk} 
is fulfilled. By the construction of $T$, $\mathcal H_0=P^2(\nu)$ and $X_0=WJ_\nu^*$, where $W$ is defined by \eqref{wwdef}. 
Recall that $\mathcal J$ is  defined by \eqref{jjkappa} and  $J_\mu$ is  defined after  \eqref{munu}. 
Applying $\mathcal J$ and taking into account \eqref{jjwwjj}, \eqref{munu},  and \eqref{jjmunustar}, 
the relation \eqref{quasidisk} can be rewritten as follows.
\begin{align*}\text{\emph{Let} } h\in H^2(\mathbb C_+), \ \text{ \emph{let} }& g\in J_\mu^*\operatorname{clos}_{L^2(\mu)}H^2(\mathbb C_+),
 \\\text{ \emph{ and let} }  &
|\{t\in \mathbb R\ :\ h(t)=g(-t)\}|>0. \ \  \text{ \emph{Then} } g\equiv 0. 
\end{align*}

We have $g\in H^2(\mathbb C_+)$. Set $g_*(z)=g(-z)$, $z\in\mathbb C_-$. Then  $g_*\in L^2(\mathbb R)\ominus H^2(\mathbb C_+)$. 
Therefore,  it sufficient to prove that $h(t)-g(-t)=0$ for a.e. $t\in \mathbb R$. 
 
By \eqref{jjmuhh2plus}, \begin{align*}g=\oplus_{n=0}^\infty \theta_\alpha^n\mathcal F^{-1}g_n, &\ \text{ where } g_n\in L^2(0,\alpha) 
\ \  \text{ and } 
\\& C_{1g}:=\sum_{n=0}^\infty \|g_n\|_{L^2(0,\alpha)}^2\omega^2_\alpha(-n-1)<\infty.\end{align*}
Since $\omega^2_\alpha(-n-1)\to\infty$ when $n\to\infty$, 
$$ C_{2g}:=\sum_{n=0}^\infty \|g_n\|_{L^2(0,\alpha)}^2<\infty.$$

Let $-\infty<b_1<b_2<+\infty$ be such that  
\begin{equation}\label{b1b2} |\{t\in (b_1,b_2)\ :\ h(t)=g(-t)\}|>0. \end{equation}
For $c>0$ define $\mathcal G(c)$ as in Theorem \ref{thmbeurling}. 

Set $\varphi(t)=h(t)-g(-t)$, $t\in (b_1,b_2)$, 
\begin{align*}f_{1n}(z)&=(\theta_\alpha^n\mathcal F^{-1}g_n)(-z) \ \text{ and }\\    
 f_n(z)& = h(z) - \oplus_{k=0}^n f_{1k}(z), \ z\in \mathbb C_+\cup \mathbb R, \ \ n\geq 0. \end{align*}
Clearly, $f_{1n}$ and $f_n$ are analytic in $\mathbb C_+$,   
$$f_{1n}(z) = \frac{1}{\sqrt{2\pi}} \mathrm{e}^{-\mathrm{i}\alpha n z}\int_0^\alpha \mathrm{e}^{-\mathrm{i} z s}g_n(s){\mathrm d}s, \ z\in \mathbb C_+,$$
and  
\begin{align*}2\pi(\varsigma_{\mathcal G(c)}(f_{1n}))^2 &=  
\sup_{0<y<c}\int_{b_1}^{b_2}|\mathrm{e}^{-\mathrm{i}(t+\mathrm{i}y)n\alpha } |^2
\Bigl|\int_0^\alpha \mathrm{e}^{-\mathrm{i} (t+\mathrm{i}y) s}g_n(s){\mathrm d}s\Bigr|^2{\mathrm d}t 
\\ &\leq \sup_{0<y<c}\int_{b_1}^{b_2}\mathrm{e}^{2yn\alpha }
\Bigl(\int_0^\alpha |\mathrm{e}^{-\mathrm{i}(t+\mathrm{i}y) s}g_n(s)|{\mathrm d}s\Bigr)^2{\mathrm d}t
\\ &\leq \mathrm{e}^{2cn\alpha }\sup_{0<y<c}
\int_{b_1}^{b_2}\Bigl(\int_0^\alpha  \mathrm{e}^{ys}|g_n(s)|{\mathrm d}s\Bigr)^2{\mathrm d}t
\\ &\leq
\mathrm{e}^{2c(n+1)\alpha}\sup_{0<y<c}\int_{b_1}^{b_2}\Bigl(\int_0^\alpha |g_n(s)|{\mathrm d}s\Bigr)^2{\mathrm d}t
\\& \leq
\mathrm{e}^{2c(n+1)\alpha}\int_{b_1}^{b_2}\alpha\Bigl(\int_0^\alpha |g_n(s)|^2{\mathrm d}s\Bigr){\mathrm d}t
\\ &=
\mathrm{e}^{2c(n+1)\alpha}(b_2-b_1)\alpha \|g_n\|_{L^2(0,\alpha)}^2.
\end{align*}
Therefore,
\begin{align*}&\varsigma_{\mathcal G(c)}(f_n)\leq \varsigma_{\mathcal G(c)}(h) + \sum_{k=0}^n\varsigma_{\mathcal G(c)}(f_{1k}) 
\\ &\leq
\varsigma_{\mathcal G(c)}(h)+ (b_2-b_1)^{1/2}\frac{\sqrt\alpha}{\sqrt{2\pi}}\sum_{k=0}^n\mathrm{e}^{c(k+1)\alpha}\|g_k\|_{L^2(0,\alpha)}
 \\ &\leq
\varsigma_{\mathcal G(c)}(h) + (b_2-b_1)^{1/2}\frac{\sqrt\alpha}{\sqrt{2\pi}}
\Bigl(\sum_{k=0}^n\mathrm{e}^{2c(k+1)\alpha}\Bigr)^{1/2}
\Bigl(\sum_{k=0}^n\|g_k\|_{L^2(0,\alpha)}^2\Bigr)^{1/2}
\\& \leq
\varsigma_{\mathcal G(c)}(h)  + (b_2-b_1)^{1/2}\frac{\sqrt\alpha}{\sqrt{2\pi}}
(n+1)^{1/2}\mathrm{e}^{c(n+1)\alpha}C_{2g}^{1/2}.
\end{align*}
Set $C_1=(b_2-b_1)^{1/2}\frac{\sqrt\alpha}{\sqrt{2\pi}}$.  Since $h\in H^2(\mathbb C_+)$, we have 
$\varsigma_{\mathcal G(c)}(h)\leq \|h\|_{H^2(\mathbb C_+)}$. 
Thus, $$\varsigma_{\mathcal G(c)}(f_n)\leq \|h\|_{H^2(\mathbb C_+)} +
C_1C_{2g}^{1/2}(n+1)^{1/2}\mathrm{e}^{c(n+1)\alpha}, \ \ \ n\geq 0.$$
Take $0<c<1/\alpha$. Then there exists $C_2$ (which depends on $c$) such that 
$$\|h\|_{H^2(\mathbb C_+)} +
C_1C_{2g}^{1/2}(n+1)^{1/2}\mathrm{e}^{c(n+1)\alpha}\leq C_2 \mathrm{e}^n \ \ \text{ for all } n\in\mathbb N.$$
 We obtain that $$ \varsigma_{\mathcal G(c)}\Bigl(\frac{1}{C_2}f_n\Bigr)\leq \mathrm{e}^n \ \text{for all } n\in\mathbb N.$$

We have \begin{align*}\Bigl(\int_{b_1}^{b_2}|\varphi(t)-f_n(t)|^2{\mathrm d}t\Bigr)^{1/2}
&\leq \|\varphi-f_n\|_{L^2(\mathbb R)}
=\|\oplus_{k=n+1}^\infty f_{1k}\|_{L^2(\mathbb R)}
\\ =\Bigl(\sum_{k=n+1}^\infty\|g_k\|_{L^2(0,\alpha)}^2\Bigr)^{1/2}
&\leq
\frac{1}{\omega_\alpha(-n-2)}\Bigl(\sum_{k=n+1}^\infty 
\omega^2_\alpha(-k-1)\|g_k\|_{L^2(0,\alpha)}^2\Bigr)^{1/2}\\&\leq \frac{1}{\omega_\alpha(-n-2)}C_{1g}\end{align*}
(because $ \omega_\alpha(-k-1)\geq \omega_\alpha(-n-2)$ for $k\geq n+1$). 

Define $M(n)$ as in Theorem \ref{thmbeurling} applying to $\frac{1}{C_2}\varphi$. Then
$$\mathrm{e}^{-M(n)}\leq \Bigl(\int_{b_1}^{b_2}\Bigl|\frac{1}{C_2}\varphi(t)-\frac{1}{C_2}f_n(t)\Bigr|^2{\mathrm d}t\Bigr)^{1/2}
\leq \frac{1}{\omega_\alpha(-n-2)}C_{1g}.$$
Therefore, $M(n)\geq \log \omega_\alpha(-n-2) -\log C_{1g}$. By  assumption \eqref{quasi},  $$\sum_{n=1}^\infty \frac{M(n)}{n^2}=\infty.$$
By Lemma \ref{leminfty},  $\frac{1}{C_2}\varphi$ satisfies the conclusion of Theorem \ref{thmbeurling}, that is, $\frac{1}{C_2}\varphi\equiv 0$. 
Therefore,  $h(t)-g(-t)= 0$ for a.e.  $t\in (b_1,b_2)$. 

Since $(b_1,b_2)$ is an arbitrary interval satisfying \eqref{b1b2}, we conclude that $h(t)-g(-t)= 0$ for a.e. $t\in \mathbb R$.  

\emph{``Only if" part.} Take $\alpha>0$. Define $w_\alpha$ by \eqref{omegarr}. If the sum in \eqref{quasi} is finite, 
then \eqref{wlogfinit} is fulfilled for $w_\alpha$.   
By Lemma \ref{lemlogconv}, there exists  $\eta\in L^\infty(\mathbb T)$ such that $\eta(\mathrm{e}^{\mathrm{i}t})=0$ for $t\in(\pi,2\pi)$, $\eta\not\equiv 0$,  and 
$\mathcal C_{\mathcal F(\eta\circ\varpi)}\in\mathcal L(L^2(\mathbb R,w_\alpha))$. By \eqref{wphi}, 
$\mathcal C_{\mathcal F(\eta\circ\varpi)}\in\mathcal L(L^2(\mathbb R, \widetilde\phi_\alpha ))$, too.
By Theorem \ref{thmconv},  $\eta\in\widehat\gamma_T(\{T\}')$ (where  $\widehat\gamma_T$ is defined in \eqref{gamma}). 
 Since $\eta\not\equiv 0$ and 
$\eta=0$ on the set of positive measure, $T$ is not quasianalytic by {\cite[Proposition 21]{ks14}}.
\end{proof}

\begin{lemma}\label{lemquasi} Let $\{a_k\}_{k=1}^\infty$ and $\{v_k\}_{k=1}^\infty$ be families of numbers such that $a_k>0$, 
$0<v_{k+1}<v_k$ for every $k$,  and $\sum_{k=1}^\infty a_k<\infty$. Let $\{k_n\}_{n=1}^\infty$ be a subsequence such that 
$$\sum_{n=1}^\infty\frac{v_{k_n}}{n}=\infty,$$ and let  $\{c_n\}_{n=1}^\infty$ be defined by the equality
$$\sum_{k=k_n+1}^\infty a_k=\mathrm{e}^{-nc_n}, \ \ \ n\geq 1.$$
For $\alpha>0$  define ${\omega^2_\alpha(-n-1)}$, $n\geq 0$, by \eqref{omegan}. If $2\alpha v_{k_n}\leq c_n$ for sufficiently large $n$, 
then \eqref{quasi} is fulfilled. 
\end{lemma}
\begin{proof} Set $a=\sum_{k=1}^\infty a_k$. We have 
\begin{align*}\frac{1}{\omega^2_\alpha(-n-1)}&=\sum_{k=1}^\infty a_k \mathrm{e}^{-2\alpha nv_k}= 
\sum_{k=1}^{k_n} a_k \mathrm{e}^{-2\alpha nv_k} + \sum_{k=k_n+1}^\infty a_k \mathrm{e}^{-2\alpha nv_k} \\ &\leq 
 \sum_{k=1}^{k_n} a_k \mathrm{e}^{-2\alpha nv_{k_n}} + \sum_{k=k_n+1}^\infty a_k 
\leq a\mathrm{e}^{-2\alpha nv_{k_n}} + \mathrm{e}^{-nc_n} \\&= a \mathrm{e}^{-2\alpha nv_{k_n}}\Bigl(1+ \frac{1}{a}\mathrm{e}^{-n(c_n-2\alpha v_{k_n})}\Bigr).
\end{align*}
Therefore, 
$$ \log \frac{1}{\omega^2_\alpha(-n-1)} \leq 
\log a -2\alpha nv_{k_n} + \frac{1}{a}\mathrm{e}^{-n(c_n-2\alpha v_{k_n})}.$$
Consequently, $$2\!\sum_{n=0}^\infty \frac{\log\omega_\alpha(-n-1)}{(n+1)^2}\geq - 
\!\sum_{n=0}^\infty\frac{\log a}{(n+1)^2} 
+2\alpha\!\sum_{n=0}^\infty\frac{nv_{k_n}}{(n+1)^2} -
\frac{1}{a}\!\sum_{n=0}^\infty\frac{\mathrm{e}^{-n(c_n-2\alpha v_{k_n})}}{(n+1)^2}.$$
By assumption, $\sum_{n=0}^\infty\frac{nv_{k_n}}{(n+1)^2} = \infty$ and
 $\sum_{n=0}^\infty\frac{\mathrm{e}^{-n(c_n-2\alpha v_{k_n})}}{(n+1)^2}<\infty$.
\end{proof}

\begin{example} Let $0<a<1$. Set $a_n=(1-a)a^{n-1}$ and $v_n=\frac{1}{\log(n+1)}$,  $n\geq 1$. 
Then $\{a_n\}_{n=1}^\infty$ and $\{v_n\}_{n=1}^\infty$ satisfy  the assumption of Lemma \ref{lemquasi}
 with $k_n=n$ and $c_n=-\log a$, $n\geq 1$.
\end{example}

\section{Existence of hyperinvariant subspaces}

Recall the definition of  a bilateral weighted shift, see \cite{est}.
Let $\omega\colon\mathbb Z \to (0,\infty)$ be a nonincreasing  function.
Set $$\ell^2_\omega  =\big \{u=\{u(n)\}_{n\in\mathbb Z}:\ \ 
\|u\|_\omega^2=\sum_{n\in\mathbb Z}|u(n)|^2\omega^2(n)<\infty\big\}.$$
The \emph{bilateral weighted shift} $S_\omega\in\mathcal L(\ell^2_\omega)$ acts by the formula 
$$(S_\omega u)(n)=u(n-1), \ \ n\in\mathbb Z,  \ \ u\in\ell^2_\omega.$$ 

\begin{theorem}\label{thmsimshift}Let $T$ be defined as in Proposition \ref{propttnu} with $\nu$ as in Lemma \ref{lemcircle}. 
For $\alpha>0$ set $$\vartheta_\alpha(z)=\mathrm{e}^{\alpha \frac{z+1}{z-1}}, \ \ \ z\in\mathbb D.$$
Set $\omega_\alpha(n)=1$, $n\geq 0$, and define $\omega_\alpha(n)$ for $n\leq -1$ by \eqref{omegan}. Then
$\vartheta_\alpha(T)$ is similar to $\oplus_{j\in\mathbb N}S_{\omega_\alpha}$.
\end{theorem}

\begin{proof} Recall that $\varpi$ and $\theta_\alpha$ are defined by  \eqref{kappa} and \eqref{thetaplane}, respectively. 
Clearly, $\vartheta_\alpha\circ\varpi=\theta_\alpha$. 
By Theorem \ref{thmconv}, $\vartheta_\alpha(T)$ is unitarily equivalent to 
 $\frac{1}{\sqrt{2\pi}}\mathcal C_{\mathcal F \theta_\alpha}$ acting on 
$L^2(\mathbb R,\widetilde\phi_\alpha)$. Define $w_\alpha$ by \eqref{omegarr}. 
By \eqref{wphi}, $\frac{1}{\sqrt{2\pi}}\mathcal C_{\mathcal F \theta_\alpha}$ is similar to the same operator acting on 
$L^2(\mathbb R,w_\alpha)$. 
We have
 \begin{align*}L^2(\mathbb R,w_\alpha)&=\{\oplus_{n\in\mathbb Z}f_n\ :\ f_n \in L^2(n\alpha,(n+1)\alpha), \\&
\ \ \ \ \sum_{n\in\mathbb Z}\|f_n\|_{L^2(n\alpha,(n+1)\alpha)}^2\omega^2_\alpha(n)<\infty\}.\end{align*}
By \eqref{shiftalpha}, $\frac{1}{\sqrt{2\pi}}(\mathcal C_{\mathcal F \theta_\alpha}f)(t)=f(t-\alpha)$, $t\in\mathbb R$.

Therefore, $ \frac{1}{\sqrt{2\pi}}\mathcal C_{\mathcal F \theta_\alpha}$ on $L^2(\mathbb R,w_\alpha)$ 
is unitarily equivalent to the bilateral shift on the weighted space of sequences
$\{f_n\}_{n\in\mathbb Z}$, where $f_n \in L^2(0,\alpha)$. Since $\dim L^2(0,\alpha)=\infty$, we conclude that 
$ \frac{1}{\sqrt{2\pi}}\mathcal C_{\mathcal F \theta_\alpha}$ on $L^2(\mathbb R,w_\alpha)$ is unitarily equivalent to 
$\oplus_{j\in\mathbb N}S_{\omega_\alpha}$.
\end{proof}

\begin{corollary}\label{cor72} Let $T$ be defined as in Proposition \ref{propttnu} with $\nu$ as in Lemma \ref{lemcircle}. 
Then for every $\alpha>0$ there exists a singular inner function $\eta_\alpha \in H^\infty$ such that 
the range of $(\eta_\alpha\circ\vartheta_\alpha)(T)$ is not dense. 
\end{corollary}

\begin{proof} Without loss of generality, we may assume that $\sum_{k=1}^\infty a_k\leq 1$. (Else, the weight 
 $\{\omega_\alpha(n)\}_{n\in\mathbb Z}$ constructed in \eqref{omegan} can be replaced by 
$\omega^2_{1\alpha}(n)=\sum_{k=1}^\infty a_k\omega^2_\alpha(n)$ for $n\leq -1$ and $\omega_{1\alpha}(n)=\omega_\alpha(n)=1$ for $n\geq 0$.)

Define the weight $\omega_\alpha=\{\omega_\alpha(n)\}_{n\in\mathbb Z}$ by \eqref{omegan}. It follows from Lemmas \ref{lemsubmult} and \ref{lemexp} 
that $\omega_\alpha$ is a dissymmetric weight (see \cite{est} for definition). By {\cite[Theorem 5.7]{est}}, 
there exists a singular inner function $\eta_\alpha \in H^\infty$ (which depends on $\omega_\alpha$) such that 
the range of $\eta_\alpha(S_{\omega_\alpha})$ is not dense. Therefore, the range of $\eta_\alpha(\oplus_{j\in\mathbb N}S_{\omega_\alpha})$
is not dense. Taking into account that $$(\eta_\alpha\circ\vartheta_\alpha)(T)=(\eta_\alpha(\vartheta_\alpha(T))$$ and 
applying Theorem \ref{thmsimshift} we obtain the conclusion of the corollary.
\end{proof}

\begin{remark}\label{remark73} Let $T$ be an operator which admits an $H^\infty$-functional calculus (see {\cite[Sec. 5]{ker4}}. 
Let $\vartheta\in H^\infty$ be a singular inner function with at least two singularities. If $\vartheta(T)$ is invertible and 
$\sigma(T)=\mathbb T$, then $T$ cannot be quasianalytic by {\cite[Theorems 2.5 and 2.6]{gam19}}. 

For $\alpha>0$, let $\vartheta_\alpha$ be defined in Theorem \ref{thmsimshift}. Then $\vartheta_\alpha$ has the only singularity at 
a point  $1\in\mathbb T$. 
 Let $T$ be defined as 
in Proposition \ref{propttnu} with $\nu$ as in Lemma \ref{lemcircle}. By Proposition \ref{propttnu}, $\sigma(T)=\mathbb T$. 
By Theorem \ref{thmsimshift}, $\vartheta_\alpha(T)$ is similar to $\oplus_{j\in\mathbb N}S_{\omega_\alpha}$. 
By \cite{est}, $S_{\omega_\alpha}$ is invertible. Thus, $\vartheta_\alpha(T)$ is invertible. By results of Sec. 6, $T$ can be quasianalytic. 
\end{remark}

\section{Convolution and Fourier transform}

The results of this section will be applied in Sec. 9. 
Recall that $\mathcal F$ and $\mathcal C_\varphi$  denote the Fourier transform and convolution with a function $\varphi $, 
see \eqref{fourier} and \eqref{defconv}.

\begin{lemma}\label{lemconv1} Suppose that $\delta>0$, $\psi\in L^1(\mathbb R)\cap C(\mathbb R)$, and 
\begin{equation}\label{intpsidelta} \int_{-\delta}^\delta\Bigl|\frac{\psi(t)-\psi(0)}{t}\Bigr|{\mathrm d}t<\infty.
\end{equation}
For $t\in\mathbb R$ set 
$$ \psi_1(t)=\int_{-\delta}^\delta\frac{\mathrm{e}^{\mathrm{i}ts}-1}{s}\psi(s){\mathrm d}s. $$
Then  $\psi_1\in L^\infty(\mathbb R)$.  
For $f\in\mathcal D(\mathbb R)$ and $t\in\mathbb R$ set 
$$ (\mathcal A_{1\psi}f)(t)=\int_{-\delta}^\delta \psi(s)\frac{f(t-s)-f(t)}{s}{\mathrm d}s. $$
Then $\mathcal A_{1\psi}f\in L^\infty(\mathbb R)$, and if $-\infty<b_1<b_2<+\infty$ are such that $f(t)=0$ for 
$t\in(-\infty,b_1]\cup[b_2,+\infty)$, then $ (\mathcal A_{1\psi}f)(t)=0$ for $t\in(-\infty,b_1-\delta]\cup[b_2+\delta,+\infty)$. 
Furthermore,
$$\mathcal F^{-1}\mathcal A_{1\psi} f=\psi_1 \mathcal F^{-1}f.$$
Consequently, $\mathcal A_{1\psi}$ can be extended from $\mathcal D(\mathbb R)$ onto $L^2(\mathbb R)$ 
and  $$\mathcal A_{1\psi}\in\mathcal L(L^2(\mathbb R)).$$  
\end{lemma}

\begin{proof} We have 
\begin{align*} |\psi_1(t)|\leq 
\int_{-\delta}^\delta\Bigl|\frac{\psi(s)-\psi(0)}{s}\Bigr||\mathrm{e}^{\mathrm{i}ts}-1|{\mathrm d}s + 
|\psi(0)|\Bigl|\int_{-\delta}^\delta\frac{\mathrm{e}^{\mathrm{i}ts}-1}{s}{\mathrm d}s\Bigl|,
\end{align*}
and \begin{align*}\int_{-\delta}^\delta\frac{\mathrm{e}^{\mathrm{i}ts}-1}{s}{\mathrm d}s=
\mathrm{i}\int_{-\delta}^\delta\frac{\sin ts}{s}{\mathrm d}s =\mathrm{i}\int_{-\delta t}^{\delta t}\frac{\sin s}{s}{\mathrm d}s.\end{align*}
Since $$\sup_{c>0}\Bigl|\int_{-c}^c\frac{\sin s}{s}{\mathrm d}s\Bigr|<\infty,$$
we conclude that $\psi_1\in L^\infty(\mathbb R)$.

Let $f\in\mathcal D(\mathbb R)$ and $t\in\mathbb R$. Then 
$$\max_{s\in[-\delta,\delta]}\Bigl|\frac{f(t-s)-f(t)}{s}\Bigr|\leq\max_{s\in[t-\delta,t+\delta]}|f'(s)|
\leq\max_{\mathbb R}|f'|<\infty.$$
Since $$|(\mathcal A_{1\psi}f)(t)|\leq\max_{s\in[-\delta,\delta]}\Bigl|\frac{f(t-s)-f(t)}{s}\Bigr|\int_{-\delta}^\delta |\psi(s)|{\mathrm d}s,$$
we conclude that $\mathcal A_{1\psi}f\in L^\infty(\mathbb R)$. 
Let $-\infty<b_1<b_2<+\infty$ be such that $f(t)=0$ for 
$t\in(-\infty,b_1]\cup[b_2,+\infty)$. Then $f(t-s)=f(t)=0$ for $t\in(-\infty,b_1-\delta]\cup[b_2+\delta,+\infty)$ and 
$s\in[-\delta,\delta]$.
Let $x\in\mathbb R$. By Fubini's theorem,  \begin{align*}
(\mathcal F^{-1}\mathcal A_{1\psi}f)(x)
&=\frac{1}{\sqrt{2\pi}}\int_{b_1-\delta}^{b_2+\delta}
\mathrm{e}^{\mathrm{i}xt}(\mathcal A_{1\psi}f)(t){\mathrm d}t 
\\ &= \frac{1}{\sqrt{2\pi}}\int_{b_1-\delta}^{b_2+\delta}\mathrm{e}^{\mathrm{i}xt}
\int_{-\delta}^\delta \psi(s)\frac{f(t-s)-f(t)}{s}{\mathrm d}s{\mathrm d}t  
\\& = \int_{-\delta}^\delta \frac{\psi(s)}{s} \frac{1}{\sqrt{2\pi}}
\int_{b_1-\delta}^{b_2+\delta}\mathrm{e}^{\mathrm{i}xt} (f(t-s)-f(t)){\mathrm d}t{\mathrm d}s
\\& = \int_{-\delta}^\delta \frac{\psi(s)}{s}(\mathrm{e}^{\mathrm{i}xs}-1) {\mathrm d}s(\mathcal F^{-1}f)(x) =\psi_1 (x)(\mathcal F^{-1}f)(x).\qedhere
\end{align*}
\end{proof}

\begin{theorem}\label{thmconv1} Let $\Psi\in L^\infty(\mathbb R)\cap C^1(\mathbb R)$ be such that 
$\Psi'\in L^1(\mathbb R)$. Set $\psi=\mathcal F\Psi'$. Suppose that $\psi\in L^1(\mathbb R)$ and 
 $\psi$ satisfies \eqref{intpsidelta} for some $\delta>0$ (and, consequently, for arbitrary finite $\delta$). 
For $f\in\mathcal D(\mathbb R)$ and $t\in\mathbb R$ set 
$$ (\mathcal A_\psi f)(t)=\int_{\mathbb R}\psi(s)\frac{f(t-s)-f(t)}{s}{\mathrm d}s. $$
Then 
$$\mathcal A_\psi f={\mathrm i}\mathcal C_{\mathcal F\Psi}f -{\mathrm i}\sqrt{2\pi}\Psi(0)f, \ \ \ f\in\mathcal D(\mathbb R).$$  
\end{theorem}

\begin{proof} Fix $\delta>0$. Set 
\begin{equation} \label{psi2} \psi_2(t)=\chi_{(-\infty,-\delta)\cup(\delta, +\infty)}(t)\frac{\psi(t)}{t}, \  \ t\in\mathbb R, 
\ \ \ \text{ and } \ \ \ c_\psi=\int_{\mathbb R}\psi_2(t){\mathrm d}t.
\end{equation} (Of course, $\psi_2$ and $c_\psi$ depend on $\delta$.)
Clearly, $\psi_2\in L^1(\mathbb R)$. Therefore, $\mathcal C_{\psi_2}f\in L^1(\mathbb R)\cap L^2(\mathbb R)$ 
 for every $f\in\mathcal D(\mathbb R)$.
 We have 
\begin{equation}\label{aasum}\mathcal A_\psi f = \mathcal A_{1\psi}f + \mathcal C_{\psi_2}f -c_\psi f, 
\ \ \ f\in\mathcal D(\mathbb R). 
\end{equation} 
Set $$\Phi(t)=\int_{\mathbb R}\frac{\mathrm{e}^{\mathrm{i}ts}-1}{s}\psi(s){\mathrm d}s,\ \ \ t\in\mathbb R.$$
By Lemma \ref{lemconv1}, $\Phi(t)\in L^\infty(\mathbb R)$ and
$\mathcal F^{-1}\mathcal A_{1\psi} f=\psi_1 \mathcal F^{-1}f$ for $f\in \mathcal D(\mathbb R)$.
Since $\psi_2\in L^1(\mathbb R)$, we have 
$$\mathcal F^{-1}\mathcal C_{\psi_2}f=\sqrt{2\pi}(\mathcal F^{-1}\psi_2)\cdot(\mathcal F^{-1}f), \ \ \ f\in\mathcal D(\mathbb R) $$
(see, for example, {\cite[Theorem VI.1.3]{katz}} or {\cite[Theorem 7.2]{rudin}}). Thus, 
\begin{equation}\label{aapsifourier} \mathcal F^{-1} \mathcal A_\psi f =\Phi  \mathcal F^{-1}f,  
\ \ \ f\in\mathcal D(\mathbb R).
\end{equation}

We will show that $\Phi'={\mathrm i}\sqrt{2\pi}\Psi'$. We have 
$$\Phi'(t)=\int_{\mathbb R}\Bigl(\frac{\mathrm{e}^{\mathrm{i}ts}-1}{s}\Bigr)'_t\psi(s){\mathrm d}s = 
\mathrm{i}\int_{\mathbb R}\mathrm{e}^{\mathrm{i}ts}\psi(s){\mathrm d}s,\ \ \ t\in\mathbb R,$$
because $\psi\in L^1(\mathbb R)$. 
Since  $L^1(\mathbb R)\subset\mathcal S'(\mathbb R)$ (see, for example, {\cite[Sec. VI.4.1]{katz}} or {\cite[Example 7.12(d)]{rudin}}), 
we have $\psi$, $\Psi'\in\mathcal S'(\mathbb R)$. Taking into account that $\psi=\mathcal F\Psi'$, we conclude that 
$$\int_{\mathbb R}\mathrm{e}^{\mathrm{i}ts}\psi(s){\mathrm d}s = \sqrt{2\pi}(\mathcal F^{-1}\psi)(t) = \sqrt{2\pi}\Psi'(t).$$
Therefore, $\Phi'=\mathrm{i}\sqrt{2\pi}\Psi'$ a.e. on $\mathbb R$. (Actually, $\Phi'(t)=\mathrm{i}\sqrt{2\pi}\Psi'(t)$ 
for every $t\in\mathbb R$, because $\Psi'\in C(\mathbb R)$ by assumption and $\Phi'=\mathrm{i}\sqrt{2\pi}\mathcal F^{-1}\psi$
with $\psi\in L^1(\mathbb R)$.)

Since $\Phi(0)=0$, we conclude that $\Phi={\mathrm i}\sqrt{2\pi}\Psi-\mathrm{i}\sqrt{2\pi}\Psi(0)$.
The conclusion of the theorem  follows from \eqref{aapsifourier} and \eqref{mmcc}.
\end{proof}

\section{Square root again}

Denote by $\varrho$ the branch of  square root defined in $\mathbb C\setminus [0,+\infty)$. 
Set \begin{equation}\label{psipsi}\Psi=\varrho\circ\varpi,\end{equation}
where $\varpi$ is defined by \eqref{kappa}. Then $\Psi$ is analytic in 
$\mathbb C\setminus\{\mathrm{i}y\ :\ |y|\geq 1\}$,
 $\Psi(\mathbb C\setminus\{\mathrm{i}y\ :\ |y|\geq 1\}) = \mathbb C_+$, and 
$$\Psi'(z)=\frac{{\mathrm i}}{\Psi(z)}\frac{1}{(z+{\mathrm i})^2}, \ \ \ z\in\mathbb C\setminus\{\mathrm{i}y\ :\ |y|\geq 1\}.$$
For $\lambda\not\in\mathbb C_+\cup\mathbb R$ set 
\begin{equation}\label{psilambda} \Psi_\lambda=\frac{1}{ \Psi-\lambda},\end{equation}
 then  $\Psi$ is analytic in 
$\mathbb C\setminus\{\mathrm{i}y\ :\ |y|\geq 1\}$ and $\Psi_\lambda'=-\frac{1}{(\Psi-\lambda)^2}\Psi'$. 

\bigskip

The proof of the following theorem can be found in {\cite[Sec. VI.7.1]{katz}}.

\begin{theorem}[Paley--Wiener]\label{thmkatzexp} Let $c>0$, and let a function $f$ be analytic in 
$\{z\in\mathbb C\ :\ z=t+\mathrm{i}y,\  t\in\mathbb R,\  y\in(-c,c)\}$ and such that 
\begin{equation}\label{katzexp}\sup_{y\in(-c,c)}\int_{\mathbb R}|f(t+\mathrm{i}y)|^2{\mathrm d}t<\infty.\end{equation}
Then $\int_{\mathbb R}\mathrm{e}^{2c|t|}|(\mathcal F(f|_{\mathbb R}))(t)|^2{\mathrm d}t<\infty$.
\end{theorem}
\medskip
\begin{lemma}\label{lempsilambda1}Let $\lambda\not\in\mathbb C_+\cup\mathbb R$, and let  $\Psi_\lambda$ be defined by
\eqref{psilambda}. Then for every $0<c<1$  $\Psi_\lambda'$ satisfies \eqref{katzexp}. 
\end{lemma}

\begin{proof} We have $|\Psi_\lambda'|\leq\frac{1}{\operatorname{dist}(\lambda,\mathbb C_+)^2}|\Psi'|$. 
Therefore, it sufficient to proof that $\Psi'$ satisfies \eqref{katzexp}, which follows from 
the equality $$|\Psi'(t+\mathrm{i}y)|^2=\frac{1}{(t^2+(1-y)^2)^{1/2}(t^2+(1+y)^2)^{3/2}}, \ \ \ t\in\mathbb R, \ \ |y|<1.  \qedhere$$
\end{proof}

\begin{lemma}\label{lempsilambda2}Let $\lambda\not\in\mathbb C_+\cup\mathbb R$, and let  $\Psi_\lambda$ be defined by
\eqref{psilambda}. Set $\psi_\lambda=\mathcal F \Psi_\lambda'$. 
Let $\delta>0$. Then $\psi_\lambda$ satisfies \eqref{intpsidelta}. 
\end{lemma}

\begin{proof} Since $\Psi_\lambda'\in L^1(\mathbb R)$, we have 
$$\frac{\psi_\lambda(t)-\psi_\lambda(0)}{t}=
\frac{1}{\sqrt{2\pi}}\int_{\mathbb R}\frac{\mathrm{e}^{-\mathrm{i}ts}-1}{t}\Psi_\lambda'(s){\mathrm d}s.$$
Therefore, \begin{align*}  
\int_0^\delta\Bigl|\frac{\psi_\lambda(t)-\psi_\lambda(0)}{t}\Bigr|{\mathrm d}t \leq
\frac{1}{\sqrt{2\pi}}\int_0^\delta\int_{\mathbb R}\Bigl|\frac{\mathrm{e}^{-\mathrm{i}ts}-1}{t}\Bigr||\Psi_\lambda'(s)|{\mathrm d}s{\mathrm d}t \\ 
\leq\frac{1}{\sqrt{2\pi}}\frac{1}{\operatorname{dist}(\lambda,\mathbb C_+)^2}\int_0^\delta
\int_{\mathbb R}\Bigl|\frac{\mathrm{e}^{-\mathrm{i}ts}-1}{t}\Bigr||\Psi'(s)|{\mathrm d}s{\mathrm d}t.
\end{align*}
We have $|\Psi'(s)|=\frac{1}{1+s^2}$, $s\in\mathbb R$, and 
 \begin{align*}\int_0^\delta
\int_{\mathbb R}\Bigl|\frac{\mathrm{e}^{-\mathrm{i}ts}-1}{t}\Bigr||\Psi'(s)|{\mathrm d}s{\mathrm d}t &=
2\int_0^\delta
\int_{\mathbb R}\frac{\bigl|\sin\frac{ts}{2}\bigr|}{|t|}\frac{1}{1+s^2}{\mathrm d}s{\mathrm d}t  \\ =
 2\int_0^\delta
\int_{\mathbb R}\frac{\bigl|\sin\frac{s}{2}\bigr|}{|t|}\frac{1}{1+(\frac{s}{t})^2}\frac{{\mathrm d}s}{|t|}{\mathrm d}t &=
2\int_{\mathbb R}\bigl|\sin\frac{s}{2}\bigr|\int_0^\delta\frac{1}{t^2+s^2}{\mathrm d}t{\mathrm d}s 
\\&=2\int_{\mathbb R}\frac{\bigl|\sin\frac{s}{2}\bigr|}{|s|}\arctan\frac{\delta}{|s|}{\mathrm d}s<\infty.
\end{align*}
Thus, $$\int_0^\delta\Bigl|\frac{\psi_\lambda(t)-\psi_\lambda(0)}{t}\Bigr|{\mathrm d}t<\infty.$$
The estimate for $\int_{-\delta}^0\Bigl|\frac{\psi_\lambda(t)-\psi_\lambda(0)}{t}\Bigr|{\mathrm d}t$ is obtained similarly. 
\end{proof}

The following two lemma are proved exactly as Lemmas \ref{lempsilambda1} and \ref{lempsilambda2}, 
therefore, their proofs are omitted. 
 
\begin{lemma}\label{lempsi1}Let  $\Psi$ be defined by
\eqref{psipsi}. Then for every $0<c<1$  $\Psi'$ satisfies \eqref{katzexp}. 
\end{lemma}

\begin{lemma}\label{lempsi2}Let  $\Psi$ be defined by
\eqref{psipsi}. Set $\psi=\mathcal F \Psi'$. Let $\delta>0$. Then $\psi$ satisfies \eqref{intpsidelta}. 
\end{lemma}

Recall that $\widehat\gamma_T$  is defined in \eqref{gamma}. 

\begin{theorem}\label{thmmain} Let $T$ be defined as in Proposition \ref{propttnu} with $\nu$ as in Lemma \ref{lemcircle}. 
 Let $\varrho$ be the branch of square root defined in $\mathbb C\setminus [0,+\infty)$. 
Then there exists $R\in\{T\}'$ such that $\widehat\gamma_T(R)=\varrho|_{\mathbb T}$, 
 and $\sigma(R)\subset\mathbb C_+\cup\mathbb R$. 
 \end{theorem}

\begin{proof} Since  $\widehat\gamma_T$ is a unital algebra-homomorphism, 
it sufficient to  prove that $\varrho|_{\mathbb T}\in\widehat\gamma_T(\{T\}')$
and $\frac{1}{\varrho-\lambda}\Big|_{\mathbb T}\in\widehat\gamma_T(\{T\}')$ 
for every $\lambda\not\in\mathbb C_+\cup\mathbb R$ (see \cite{ks14}).

Let $\Psi$ be defined by \eqref{psipsi}, and for $\lambda\not\in\mathbb C_+\cup\mathbb R$ let $\Psi_\lambda$
be defined by \eqref{psilambda}. Take $\alpha>0$. Let $\widetilde\phi_\alpha$ be defined in Theorem \ref{thmconv}. 
By  Theorem \ref{thmconv}, it is sufficient to prove that $\mathcal C_{\mathcal F\Psi}\in\mathcal L(L^2(\mathbb R,\widetilde\phi_\alpha))$, 
and  $\mathcal C_{\mathcal F\Psi_\lambda}\in\mathcal L(L^2(\mathbb R,\widetilde\phi_\alpha))$ for every $\lambda\not\in\mathbb C_+\cup\mathbb R$.  
We may assume that $\sum_{k=1}^\infty a_k\leq 1$ (where $\{a_k\}_{k=1}^\infty$ are from the construction of $\nu$).
  Let $w_\alpha$ be defined by \eqref{omegarr}. By \eqref{wphi}, 
$w_\alpha\asymp \widetilde\phi_\alpha$. Therefore, it sufficient to prove that 
$\mathcal C_{\mathcal F\Psi}\in\mathcal L(L^2(\mathbb R,w_\alpha))$, 
and $\mathcal C_{\mathcal F\Psi_\lambda}\in\mathcal L(L^2(\mathbb R,w_\alpha))$
  for every $\lambda\not\in\mathbb C_+\cup\mathbb R$.  
Since $\sup_{(b_1,b_2)}w_\alpha<\infty$ for every $-\infty<b_1<b_2<+\infty$, we have 
$\mathcal D(\mathbb R)\subset L^2(\mathbb R,w_\alpha)$. 

Take $0<\delta<\alpha$.  Set $\psi=\mathcal F\Psi'$. By Theorem \ref{thmkatzexp} and Lemmas \ref{lempsi1} and \ref{lempsi2}, 
$\Psi$ satisfies Theorem \ref{thmconv1}. By Theorem \ref{thmconv1}, it sufficient to prove that 
the mapping $\mathcal A_\psi$ which is defined on $\mathcal D(\mathbb R)$ can be extended as a (linear, bounded) operator 
onto $L^2(\mathbb R,w_\alpha)$. 
  
By \eqref{aasum}, it sufficient to prove that $\mathcal A_{1\psi}$ and 
$\mathcal C_{\psi_2}$, where $\psi_2$ is defined by \eqref{psi2}, 
can be extended as  (linear, bounded) operators 
onto $L^2(\mathbb R,w_\alpha)$.  By Lemma \ref{lemsubmult}, $w_\alpha$ satisfies the assumptions 
of Lemma \ref{lemsupport}. By Lemma \ref{lemconv1},  
$\mathcal A_{1\psi}$ satisfies the assumptions of Lemma \ref{lemsupport}. 
Thus, $\mathcal A_{1\psi}\in\mathcal L(L^2(\mathbb R,w_\alpha))$ by Lemma \ref{lemsupport}.
 By Theorem \ref{thmkatzexp} and Lemma \ref{lemexp}, $\psi_2 \sqrt{w_\alpha}\in L^1(\mathbb R)$.
By Corollary \ref{corsubmult}, $w_\alpha$ satisfies the assumptions 
of Lemma \ref{lemconvsubmult}.  By Lemma \ref{lemconvsubmult}, 
$\mathcal C_{\psi_2}\in\mathcal L(L^2(\mathbb R,w_\alpha))$.  

Thus, $\mathcal C_{\mathcal F\Psi}\in\mathcal L(L^2(\mathbb R,\widetilde\phi_\alpha))$. 
For $\Psi_\lambda$ with $\lambda\not\in\mathbb C_+\cup\mathbb R$, the proof is the same. \end{proof}

\begin{corollary}\label{cormain} There exists a quasianalytic contraction $R$ 
with  $\sigma(R)=\{\mathrm{e}^{\mathrm{i}t}\ : \ t\in [0,\pi]\}$ and with a unitary asymptote $U_{\sigma(R)}$. 
Consequently, $\sigma(R)$ coincides with the quasianalytic spectral set of $R$ and $\sigma(R)\neq \mathbb T$. 
\end{corollary}
\begin{proof} Let $T$ and $R$ be from  Theorem \ref{thmmain}. Clearly,  $\bigl(\varrho|_{\mathbb T}\bigr)^2=\chi$ (where $\chi(z)=z$, $z\in\mathbb T$).
Since $\widehat\gamma_T(R)=\varrho|_{\mathbb T}$ and $\widehat\gamma_T(T)=\chi$, we conclude that $R^2=T$, 
because $\widehat\gamma_T$ is a unital algebra-homomorphism. 

Set $\sigma=\{\mathrm{e}^{\mathrm{i}t}\ : \ t\in [0,\pi]\}$. By Theorem \ref{thmmain}, 
$\sigma(R)\subset\mathbb C_+\cup\mathbb R$. By Proposition \ref{propttnu}, $\sigma(T)=\mathbb T$. 
Since $R^2=T$, we have $\sigma(T)=\{\lambda^2\ :\ \lambda\in\sigma(R)\}$. Therefore, $\sigma(R)\subset\mathbb T$. 
Consequently, $$\sigma(R)\subset\mathbb T\cap( \mathbb C_+\cup\mathbb R)= \sigma.$$ 

By Proposition \ref{propttmain}, $T$ is similar to a contraction. By Theorem \ref{thmsquare}, $R$ is similar to a contraction, too.
By Lemma \ref{lemsquare}, $\varrho(U_{\mathbb T})$ is a unitary asymptote of $R$. Since $\varrho(U_{\mathbb T})$ is unitarily equivalent to 
$U_\sigma$, we have that $U_\sigma$ is a unitary asymptote of $R$. By \eqref{sigmagamma} applied to $R$,  $\sigma\subset\sigma(R)$. Thus, $\sigma(R)=\sigma$. 

Let $\nu$ be chosen such that $T$ is quasianalytic. 
(It is possible by results of Sec. 6.) By Corollary \ref{corsquare}, $R$ is quasianalytic.
\end{proof}

\begin{remark} Let $T$ and $R$ be operators constructed  in the proof of Corollary \ref{cormain}. 
By Proposition \ref{propttmain} and Lemma \ref{lemsquare},  unitary asymptotes of $T$ and  $R$  are cyclic unitary operators. 
Therefore, $\{T\}'$ and  $\{R\}'$ are abelian algebras. Since $R\in\{T\}'$, we conclude that $\{R\}'=\{T\}'$ 
by {\cite[Proposition 11]{ks14}}. Consequently, $\operatorname{Hlat}R=\operatorname{Hlat}T$. By Corollary \ref{cor72}, 
$\operatorname{Hlat}T$ is nontrivial. Therefore, $\operatorname{Hlat}R$ is nontrivial, too. 
 \end{remark}

\begin{remark} Let $T$ be defined as in Proposition \ref{propttnu} with $\nu$ as in Lemma \ref{lemcircle}. For $\alpha>0$, 
let $\vartheta_\alpha$ be defined in Theorem \ref{thmsimshift}. By Remark \ref{remark73}, 
 $\vartheta_\alpha(T)$ is invertible.

For every $0<r<1$ set $\mathcal G_r=\mathbb D\setminus D_r$, where $D_r$ is defined in \eqref{circle}. 
Then $\inf_{\mathcal G_r}|\vartheta_\alpha|>0$. 
Denote by $\kappa_r$ a conformal mapping of  $\mathbb D $ onto $\mathcal G_r$. 
Let $Q$ be an operator  which admits an $H^\infty$-functional calculus (see {\cite[Theorem 23]{ker4}}). 
 By \cite{ker15} and \cite{gam19}, $\vartheta_\alpha(\kappa_r(Q))$ is invertible. 

A question appears: whether exists an operator $Q_r$ such that $T=\kappa_r(Q_r)$? 
It is possible to prove that there exists $Q_{1r}\in\{T\}'$ such that 
$\widehat\gamma_T(Q_{1r})=\kappa_r^{-1}$ and $\sigma(Q_{1r})$ is a proper subarc of $\mathbb T$. 
 But the estimate obtained by the author is $\|Q_{1r}^n\|\leq Cn(\log n)^2$ for sufficiently large $n\in\mathbb N$, 
which does  not allow to define $\kappa_r(Q_{1r})$.
\end{remark}

\end{document}